\documentclass[12pt]{article}
\usepackage{amsfonts,amsmath,amsthm, latexsym, amssymb, mathrsfs, graphicx, float}
\usepackage[dvipsnames]{xcolor}
\usepackage{bm, wasysym, tikz}
\usepackage{enumerate}
\usepackage{datetime}
\usetikzlibrary{arrows, shapes}
\tikzset{>=stealth'}
\usepackage{pstricks,pst-node,pst-text,pst-3d}
\usepackage{color}
\usepackage[polish,english]{babel}
\usepackage[utf8]{inputenc}
\usepackage{polski}
\def\beqq{\arraycolsep1pt\begin{eqnarray*}}
\def\eeqq{\end{eqnarray*}}

\newcommand{\R}{{\mathbb R}}

\newcommand{\dist}{{\rm{dist}}}

\begingroup
\newtheorem{theorem}{Theorem}[section]
\newtheorem*{theorem*}{Theorem}
\newtheorem{lemma}[theorem]{Lemma}
\newtheorem{proposition}[theorem]{Proposition}
\newtheorem{corollary}[theorem]{Corollary}
\newtheorem{definition}[theorem]{Definition}
\newtheorem*{remark*}{Remark}

\endgroup

\newcommand{\cB}{\mathcal{B}}
\newcommand{\cD}{\mathcal{D}}

\newcommand{\cH}{\mathcal{H}}

\newcommand{\cO}{\mathcal{O}}

\newcommand{\cR}{\mathcal{R}}
\newcommand{\cS}{\mathcal{S}}
\newcommand{\cT}{\mathcal{T}}
\newcommand{\cU}{\mathcal{U}}

\newcommand{\bR}{\mathbb{R}}

\newcommand{\ib}{\mathrm{i}_{\mathrm{B}}}
\newcommand{\db}{\mathrm{d}_{\mathrm{B}}}

\newcommand{\morse}{\mathrm{m}^-}

\title{Bifurcation of closed orbits from equilibria of Newtonian systems with Coriolis forces}

\author{Anna Go{\l}\c{e}biewska, Ernesto P\'erez-Chavela, S{\l}awomir Rybicki\\ and Antonio J. Ure\~na}

\date{\today}

\begin{document}

\maketitle

\begin{abstract}
We consider autonomous Newtonian systems with Coriolis forces in two and three dimensions and study the existence of branches of periodic orbits emanating from equilibria. We investigate both degenerate and nondegenerate situations. While Lyapunov’s center theorem applies locally in the nondegenerate, nonresonant context, equivariant degree theory provides a global answer which is significant also in some degenerate cases. 

We apply our abstract results to a problem from Celestial Mechanics.
More precisely, in the three-dimensional version of the Restricted Triangular Four-Body Problem with possibly different primaries our results show the existence of at least seven branches of periodic orbits emanating from the stationary points. 

\end{abstract}

\section{Introduction}

Many mechanical problems can be modeled as the motion of a particle subjected to the conservative force  created by a uniformly rotating source. This is for instance the case of the classical circular restricted three-body problem or, more generally, the restricted problem consisting of a massless particle moving under the attraction of a given family of primaries revolving solidly around the origin. In rotating coordinates we obtain a system of autonomous second-order equations which in the {\em planar} case reads
\begin{equation*}
	\begin{cases}
\ddot x+2\dot y+\frac{\partial V}{\partial x}(x,y)=0,\\
\ddot y-2\dot x+\frac{\partial V}{\partial y}(x,y)=0,
\end{cases}
\tag*{\bf [2d]}
\end{equation*}
and in the {\em spatial} context becomes 
\begin{equation*}
	\begin{cases}
	\ddot x+2\dot y+\frac{\partial V}{\partial x}(x,y,z)=0,\\
	\ddot y-2\dot x+\frac{\partial V}{\partial y}(x,y,z)=0,\\
	\ddot z+\frac{\partial V}{\partial z}(x,y,z)=0,
	\end{cases}
	\tag*{\bf [3d]}
\end{equation*}
(the rotation is assumed to take place in the $xy$ plane). One can condense both situations in a unified way  as follows:
\begin{equation}\label{soe}
	\ddot q-2\alpha_N\dot q+V'(q)=0,\qquad q\in\Omega,
\end{equation}
where $q=(x,y)$  or $q=(x,y,z)$ is a variable in $\mathbb R^N$ with $N=2$ or $N=3$, respectively. We assume that the {\em effective potential} $V=V(q)$ is defined and has class $C^2$ on some open domain $\Omega\subset\mathbb R^N$, its gradient being denoted by $V':\Omega\to\mathbb R^N$. Moreover, the skew-symmetric matrices $\alpha_2\in\mathbb R^{2\times 2}$, $\alpha_3\in\mathbb R^{3\times 3}$ are given by
$$\alpha_2=\left(\begin{array}{rr}0&-1\\ 1&0\end{array}\right),\qquad\alpha_3=\left(
\begin{array}{c|c}
	\raisebox{-8pt}{\normalsize\mbox{{$\alpha_2$}}} & 0\\[-1.1ex]  & 0\\\hline\\[-2ex]  0\ \  0&0
\end{array}
\right)\,.$$

In the spatial case {\bf [3d]} we shall further assume that $V=V(x,y,z)$ satisfies
	\begin{equation*}
		\frac{\partial V}{\partial z}(x,y,0)=0,\qquad (x,y)\in\widetilde\Omega\,,\tag*{\bf [H$_1$]}
	\end{equation*}
where $\widetilde\Omega:=\{(x,y)\in\mathbb R^2:(x,y,0)\in\Omega\}$. In this way, the restriction
$$\widetilde V(x,y):=V(x,y,0),\qquad (x,y)\in\widetilde\Omega,$$ lies under the framework established in {\bf [2d]}, and the corresponding closed orbits are actually closed orbits of the full spatial problem.

\smallskip

The equilibria of \eqref{soe} coincide with the critical points of $V$. Throughout this paper it will always be assumed that:

\begin{equation*}
\text{All critical points of $V$ are isolated.}\tag*{\bf [V]}
\end{equation*}

\smallskip

\noindent Thus, we allow the possibility that $V$ has infinitely many critical points, but in this case they must accumulate either at infinity or on the boundary  of $\Omega$.
In addition, in the spatial case {\bf [3d]} we shall always assume that 
\begin{equation*}
q_0=(\widetilde q_0,0)\in\widetilde\Omega\times\{0\},\ \beta_3(q_0):=\frac{\partial^2 V}{\partial z^2}(q_0)>0,\ \ \text{ for every }q_0\in(V')^{-1}(0)\,.\tag*{\bf [H$_2$]}\end{equation*}
In combination with {\bf [H$_1$]}, the second part of {\bf [H$_2$]} roughly means that the force $-V'$ attracts our particle towards $\widetilde\Omega\times\{0\}$ when it gets close to an equilibrium.
 The goal of this paper is to study the {\em  existence of global branches of closed orbits of (\ref{soe}) emanating from  equilibria}, both in the planar and the spatial cases.

\smallskip

A well-known {\em necessary} condition for the existence of closed orbits near a given equilibrium $q_0$ is that it must be nonhyperbolic. On the other hand,  a {\em sufficient} condition is provided by  Lyapunov's center theorem: under some nondegeneracy and nonresonance conditions, the existence of emanating families of closed orbits is guaranteed in the elliptic and elliptic-hyperbolic cases. Perhaps unsurprisingly, these situations can be described in terms of the spectrum of the Hessian matrix $V''(q_0)$. In addition, the equivariant degree theory for Hamiltonian systems provides powerful tools which will allow us to {\em (a):} replace the nondegeneracy condition with a milder condition on the Brouwer index of $q_0$ as a zero of $V'$;  {\em (b):} obtain the existence of {\em global} (rather than local) branches.

\smallskip

These results are formulated in a precise form in Sections \ref{sec:def}-\ref{sec:results}. In the planar setting {\bf [2d]} we characterize the existence of bifurcating branches of closed orbits in terms of the eigenvalues $\beta_1,\beta_2$ of $V''(q_0)$. The degenerate situations where $(\beta_1,\beta_2)$ lies on the boundary of the existence region while avoiding the boundary of the nonexistence one can still be dealt with if the  Brouwer index of $V'$ at the isolated zero $q_0$ does not vanish. On the other hand, in the spatial case {\bf [3d]} these bifurcating branches of closed orbits do exist  {\em in all nondegenerate cases}, and even in the degenerate ones when the corresponding Brouwer index does not vanish. Moreover, in some situations the bifurcating branch can be shown to be nonplanar. We illustrate our results with an example and formulate a number of open questions. 

\smallskip

In Section \ref{sec:celmec} we apply our abstract results to a restricted four-body problem. Assuming that three (possibly different) positive masses revolve around their center of masses in a Lagrangian equilateral triangle, we study the motion of a massless test particle subjected to their gravitational attraction.  For the three-dimensional problem it turns out that there are at least seven branches of closed orbits emanating from seven corresponding libration points.

\smallskip

In Section \ref{sec:branches} we state, without proof, a bifurcation result taken from \cite{DanRyb} which will be the main tool behind of our arguments. Based on degree theory for equivariant gradient maps, it applies to general (autonomous) Hamiltonian systems having a given equilibrium. Under the assumption that the Morse index of certain $8\times 8$ or $12\times12$ matrices $S_T$ changes as the parameter $T$ varies on $(0,+\infty)$, Theorem \ref{th:bifurcation} states the existence of a branch of closed orbits emanating from the equilibrium. We also recall De Gua's corollary of Descartes' rule of signs \cite{Gua}, which allows the exact computation of the Morse index of a symmetric matrix by looking at the number of sign changes in the sequence of coefficients of its characteristic polynomial. 

\smallskip

Sections \ref{sec:hamiltonian}-\ref{sec:prf}  are devoted to prove the general results announced in Section \ref{sec:results}. More precisely, in Section \ref{sec:hamiltonian} we observe that \eqref{soe}  can be rewritten in a Hamiltonian form and compare some properties that a given equilibrium may have in both contexts. Section \ref{sec:spectrum} is devoted to study the spectrum of the matrices $S_T$, and in Section \ref{sec:Tk} we compute their Morse index. This will lead us to complete the proofs in Section \ref{sec:prf}. The paper closes with an Appendix where we discuss a couple of elementary facts from linear algebra needed in our computations of Sections \ref{sec:spectrum}-\ref{sec:Tk}.

\smallskip

Sumarizing, Sections \ref{sec:def}-\ref{sec:results} are devoted to present the general bifurcation results, in Section \ref{sec:celmec} we apply them in a problem of Celestial Mechanics, the expository Section \ref{sec:branches} collects a couple of known theorems to be used later, and Sections \ref{sec:hamiltonian}-\ref{sec:prf} are concerned with the proofs. In addition, this latter `proof block' resorts from time to time to material from the Appendix. The purpose of this scheme is to motivate the theory before going into the mathematical details, while keeping at the same time the pace of the exposition.

\smallskip

Predicting the existence of closed orbits near an equilibrium is a classical problem which goes back to Poincar\'e \cite{Poi}. Coinciding with the space race, the seventies saw a renewed interest in this question \cite{AleYor,MeySch,Mos,Rabinowitz2,Sch,Wei} that continues to this day \cite{Bartsch,Ben,DanRyb,Dui,GarIze,GebMar,MacRyb, MeyPalYan,RadzRyb,Szu} (the lists are just an small sample and far from complete). The fact that  the set of closed orbits is invariant under translations in the time variable introduces a degeneracy in the problem that has been solved in a variety of ways,  including: geometrical index theories, simplectic reduction techniques, equivariant Ljusternik-Schnirelmann category theory, equivariant Morse theory and equivariant topological degree theory.  Our choice of the latter is motivated by the fact that it both applies to {\em degenerate} situations and produces {\em global} branches of closed orbits.
  
\section{Equilibria and branches of closed orbits}\label{sec:def}
Throughout this paper we study periodic solutions of (\ref{soe}) whose periods are not fixed in advance and, in general, may differ from one solution to another. The set of  positive periods of a given periodic solution  $q=q(t)$ is an infinite  semigroup of real numbers. We shall follow the convention of counting each periodic solution infinitely many times -once for each positive period. After identifying the periodic solution $q:\mathbb R/T\mathbb Z\to\mathbb R^{N}$ with the pair $(T,\bar q)$ where $\bar q(\theta)=q\big(\frac{T}{2\pi}\theta\big)$, we may see them  as elements of the cartesian product $(0,+\infty)\times C(\mathbb S^1,\Omega)$. Here, and in what follows, {  $\mathbb S^1=\{\theta:\theta\in\mathbb R/2\pi\mathbb Z\}$. It motivates the following terminology: the pair $(T,\bar q)\in (0,+\infty)\times C(\mathbb S^1,\Omega)$ will be called a {\em closed orbit} if $q(t):=\bar q\big(\frac{2\pi}{T}t\big)$ is a periodic solution.  
	
	\smallskip
	
	For instance, if $q_0\in\Omega$ is a critical point of $V$ (or equivalently, an equilibrium of (\ref{soe})), then $(T,q_0)$ is a closed orbit for each $T>0$ -by abuse of notation we still denote by $q_0$ to the corresponding constant map defined on $\mathbb S^1$. The closed orbits obtained in this way will be  considered {\em trivial} and therefore the set of trivial closed orbits can be canonically identified with $(0,+\infty)\times(V')^{-1}(0)$.

\smallskip

On the other hand, a closed orbit $(T,\bar q)$ will be called {\em nontrivial} if the function $\bar q:\mathbb S^1\to\mathbb R^N$ is nonconstant. We denote by $\Lambda$ the closure in $\mathbb R\times C(\mathbb S^1,\Omega)$ of the set of nontrivial closed orbits.  By a result of Yorke \cite{Yor}, the set $\Lambda$ is contained into $(0,+\infty)\times C(\mathbb S^1,\Omega)$. Therefore, all pairs $(T,\bar q)\in\Lambda$ are closed orbits but some trivial closed orbits may not belong to $\Lambda$. It motivates the notion of {\em branch}, a word that has been used quite loosely above. The following definition is inspired by Leray-Schauder's Continuation Theorem \cite{LerSch} (see also \cite{Rab}). 

\begin{definition}[Branch of closed orbits]\label{branch} A connected component $\cB$ of $\Lambda$ will be called a branch of closed orbits if it satisfies at least one of the following assumptions: 
	\begin{enumerate}
		\item[(a)] $\cB$ is unbounded in the Banach space $\mathbb R\times C(\mathbb S^1,\mathbb R^{N})$,
		\item[(b)] the closure of the set $\{\bar q(\theta):\theta\in\mathbb S^1,\ (T,\bar q)\in\cB\}$ has nonempty intersection with $\partial\Omega$, or
		\item[(c)]  $\cB$ is compact with the inherited $\bR\times C(\mathbb S^1,\bR)$-topology and contains at least two different trivial closed orbits.
	\end{enumerate}

The branch of closed orbits $\mathcal B$ is said to emanate from the trivial closed orbit $(T,q_0)\in (0,+\infty)\times(V')^{-1}(0)$ provided that $(T,q_0)\in\mathcal B$. We shall simply say that $\mathcal B$ emanates from the equilibrium $q_0$ provided that $(T,q_0)\in\mathcal B$ for some $T>0$.

\end{definition}

For short, one can name the three possibilities in Definition \ref{branch} by saying that a branch must, either:  {\em (a) be unbounded}, or {\em (b) go up to the boundary}, or  {\em (c) be compact, containing at least two trivial closed orbits}. A couple of remarks are in order here:

\begin{enumerate}
\item[(I)] Possibility {\em (a)} may happen  if either $\{\bar q(\theta):\theta\in\mathbb S^1,\ (T,\bar q)\in\cB\}$ is unbounded in $\bR^N$ or  the set of periods $\{T:(T,\bar q)\in\cB\}$  is unbounded from above. The second situation has been called the {\em blue sky catastrophe} in the literature. It can be excluded in some special cases, see \cite[\S 7.6]{FraVan}.
\item[(I\hspace{-0.03cm}I)]For the purposes of this paper, case {\em (c)} could have been stated in a stronger form: not only $\cB$ must be compact and contain at least two trivial closed orbits, but in addition the sum of the bifurcation numbers of the trivial closed orbits  in $\cB$ must be zero. The {\em bifurcation number} is an integer  $\gamma_N(T,q_0)\in\mathbb Z$ associated to any trivial closed orbit $(T,q_0)\in(0,+\infty)\times (V')^{-1}(0)$. More generally it can be defined for trivial closed orbits of arbitrary Hamiltonian systems, see \eqref{bb3} in Section \ref{sec:branches}. For system (\ref{soe}) these integers can be computed explicitly in terms of $T$ and the spectrum of $V''(q_0)$ (and also the Brouwer index $\ib(q_0,V')$ in the degenerate cases). See \eqref{bn2}-\eqref{bn3} in Section \ref{sec:prf}. See also e.g. \cite{Brown} for an introduction to the Brouwer index.
\end{enumerate}
\section{The general results}\label{sec:results}
 The main purpose of this section is to announce the main abstract bifurcation results of the paper. They are organized in Theorems \ref{thm:main2d}-\ref{thm:main3d}. We shall subsequently apply them in a problem of Celestial Mechanics (in Section \ref{sec:celmec}) and devote Sections \ref{sec:hamiltonian}-\ref{sec:prf} to their proofs.
 
 \smallskip
 
Fix an equilibrium $q_0\in\Omega$. {\em We shall always denote by $\beta_1,\beta_2$ the eigenvalues of $V''(q_0)$ in the planar case {\bf [2d]}, or the eigenvalues of $\widetilde V''(\widetilde q_0)$ in the spatial situation  {\bf [3d]}}. Inside the plane of couples $(\beta_1,\beta_2)\in\mathbb R^2$ we consider the closed set $$C:=(\mathbb R\times\{0\})\cup(\{0\}\times\mathbb R)\,,$$ made up of both coordinate axis, and the open set
 $$\cR_0:=\big\{(\beta_1,\beta_2)\in(-\infty,0)^2:\beta_1+\beta_2<\max(-4,-2-(\beta_1-\beta_2)^2/8)\big\}\,.$$
 
 \smallskip

 These sets are pictured in Figure 1 below. We have also labeled the four domains $\cR_i$, $1\leq i\leq 4$, in which the complementary open set $\mathbb R^2\backslash(\bar{\cR}_0\cup C)$ is divided. Notice that $\cR_1$, $\cR_2$ and $\cR_4$ coincide, respectively, with the first, second, and fourth open quadrants, while $\cR_3$ is the set of points $(\beta_1,\beta_2)$ in the open third quadrant $(-\infty,0)^2$ such that $\beta_1+\beta_2>\max(-4,-2-(\beta_1-\beta_2)^2/8)$.
 \begin{figure}
 	\begin{center}
 		\begin{tikzpicture}[xscale=0.5, yscale=0.5]
 			\draw[red,very thick](-6,0)--(-4,0);
 			\draw[red,very thick](0,-6)--(0,-4);
 			\draw[olive,style= {line width=0.6pt}](-6,0)--(-4,0);
 			\draw[olive,style= {line width=0.6pt}](0,-6)--(0,-4);
 			\draw[olive,dashed,very thick](-4,0)--(3.5,0);
 			\draw[olive, dashed,very thick](0,-4)--(0,3.5);
 			\draw[help lines,dashed,very thin,->](0,-5)--(0,4);
 			\draw[help lines,dashed,very thin,->](-5,0)--(4,0);
 			\draw[red, very thin, densely dashed](-4,-5.5)--(-0.75,-2.25);
 			\draw[red, very thin, densely dashed](-2,-6)--(-0.5,-4.5);
 			\draw[red, very thin, densely dashed](-5.5,-4)--(-2.25,-0.75);
 			\draw[red, very thin, densely dashed](-6,-2)--(-4.5,-0.5);
 			\draw[red, very thick, domain=(-4:4)] plot({\x/2-\x*\x/16-1,-\x/2-\x*\x/16-1});
 			\node (A) at(-3.3,-3.3){\color{red}\small $\cR_0$};
 			\node[right] at (0,3.8){\color{gray}\scriptsize$\beta_2$};
 			\node[below] at (4,0){\color{gray}\scriptsize$\beta_1$};
 			\node[below] at (-4,0){\color{gray}\scriptsize$-4$};
 			\node[right] at (0,-4){\color{gray}\scriptsize$-4$};
 			\node[right] at (2,0.5){\color{olive}\footnotesize $C$};
 			\node at(-0.6,-0.6){\color{cyan}\small $\cR_3$};
 			\node (B) at(-3.1,2.5){\color{blue}\small $\cR_2$};
 			\node (C) at(2.5,-3.1){\color{brown}\small $\cR_4$};
 			\node at(1.9,2.2){\color{violet}\small $\cR_1$};
 		\end{tikzpicture}
 	\end{center}
 	\caption{The plane $(\beta_1,\beta_2)$ of eigenvalues of $ V''(q_0)$ (in the planar case) or $ \widetilde V''(\widetilde q_0)$ (in the spatial case). }
 \end{figure}
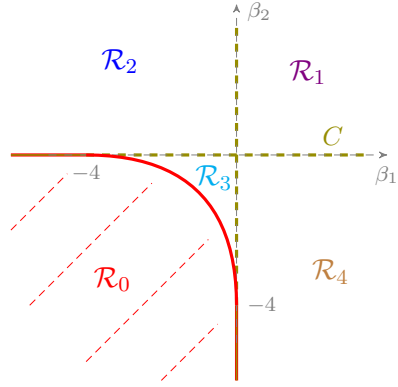
 
 \smallskip
 We also consider the continuous, symmetric functions $$T_-:(\bR^2\backslash\bar\cR_0)\cup((\partial\cR_0)\backslash C)\to(0,+\infty),\qquad T_+:\cR_1\cup\cR_3\cup((\partial\cR_0)\backslash C)\to(0,+\infty),$$ defined by
 \begin{equation}\label{lpm}
 	T_\pm(\beta_1,\beta_2):=2\pi\sqrt{2}\left(\beta_1+\beta_2+4\mp 2\sqrt{2}\sqrt{\beta_1+\beta_2+2+(\beta_1-\beta_2)^2/8}\right)^{-1/2}\,.
 \end{equation}
 
It can be checked that $T_-(\beta_1,\beta_2)<T_+(\beta_1,\beta_2)$ on $\cR_1\cup\cR_3$ and $T_-(\beta_1,\beta_2)=T_+(\beta_1,\beta_2)$ on $(\partial\cR_0)\backslash C$, see Lemma \ref{lem1234321}.  Our interest in these functions comes from the fact that they can be used to describe the nonzero, purely-imaginary part of the spectrum  of the linearized system (\ref{soe}) at the equilibrium $q_0$. In fact, in the planar case  {\bf [2d]} these characteristic exponents are $-i\,2\pi/T_\pm(\beta_1,\beta_2)$, $i\,2\pi/T_\pm(\beta_1,\beta_2)$, wherever defined. %No purely-imaginary characteristic exponents exist if $(\beta_1,\beta_2)\in\cR_0$. 
 In addition, in the spatial situation  {\bf [3d]} it suffices to add $\pm\sqrt{\beta_3}i$ to the previous lists (see Lemma \ref{lem6}). We arrive to one of the main general results of this paper:
 
 \smallskip

\begin{theorem}\label{thm:main2d}
Assume {\bf [2d]}. The following hold:
\begin{enumerate}
	\item[(i)]
	If $(\beta_1,\beta_2)\in \cR_0$, then there are not closed orbits in a sufficiently small neighborhood of the equilibrium $q_0$.
\item[(ii)]  If $(\beta_1,\beta_2)\not\in\bar\cR_0\cup C$,  then  there is a branch of closed orbits emanating from $\left(T_-(\beta_1,\beta_2),q_0\right)$. The same conclusion holds if $(\beta_1,\beta_2)\in C\backslash\bar\cR_0$ and $\ib(q_0, V') \not = 0$.
\item[(iii)]  If $(\beta_1,\beta_2)\in\cR_1\cup\cR_3$,  then  there is also a branch of closed orbits emanating from $\left(T_+(\beta_1,\beta_2),q_0\right)$. 
\end{enumerate}	
\end{theorem}

The formulation of {\em (i)} is conveniently simple but needs some interpretation. It could be more precisely expressed as follows: 	If $(\beta_1,\beta_2)\in \cR_0$ then there exists an open set $\Omega_1\subset\Omega$ with $q_0\in\Omega_1$ such that if $(T,\bar q)\in(0,+\infty)\times C(\mathbb S^1,\Omega_1)$ is a closed orbit of \eqref{soe} then $\bar q(\theta)=q_0$ for every $\theta\in\mathbb S^1$.  
 
 \smallskip
 
 Concerning the second part of {\em (ii)} we point out that if  $(\beta_1, \beta_2)\notin C$ then  
 \begin{equation}\label{bd}
 	\ib(q_0,V')={\rm sign}(\beta_1\beta_2)=\begin{cases}
 		1&\text{if }(\beta_1,\beta_2)\in(\bar\cR_0\backslash C)\cup\cR_1\cup\cR_3,\\
 		-1&\text{if }(\beta_1,\beta_2)\in\cR_2\cup\cR_4.
 	\end{cases}
 \end{equation}
 Thus, assertion {\em (ii)} of Theorem \ref{thm:main2d} could be reformulated as follows: if $(\beta_1,\beta_2)\not\in\bar\cR_0$ and $i_B(q_0,V')\not=0$,  then there is a branch of closed orbits of \eqref{soe} emanating from $(T_-(\beta_1,\beta_2),q_0)$. 
 Two sufficient conditions implying, each of them, the inequality $\ib(q_0,V')\not=0$, are: {\em (a):} $V$ has a {\em local extremum} at $q_0$ (\cite[Lemma 6.5]{Kra}, \cite{Rab3}), or {\em (b):} $V$ is {\em even} with respect to $q_0$
(Borsuk-Ulam theorem). This observation leads to the following

\begin{corollary}\label{corol1XT}{If  $(\beta_1,\beta_2)\not\in\bar\cR_0$ and either $V$ attains a local extremum at $q_0$, or $V$ is even with respect to $q_0$, then there is a branch of closed orbits of \eqref{soe} emanating from $\left(T_-(\beta_1,\beta_2),q_0\right)$.}
\end{corollary}

In Section \ref{sec:spectrum} it will be observed that if $(\beta_1,\beta_2)\in\cR_0$ then the equilibrium $q_0$ is hyperbolic, if $(\beta_1,\beta_2)\in\cR_1\cup\cR_3$ the equilibrium is elliptic, and if $(\beta_1,\beta_2)\in\cR_2\cup\cR_4$ then the equilibrium is of elliptic-hyperbolic type (see Lemmas \ref{lem1234321}-\ref{lem6}). Therefore, the branches of closed orbits described in {\em (ii)} and {\em (iii)} are the global continuations of the {\em short} and {\em long period}  Lyapunov families, respectively. In the elliptic case $(\beta_1,\beta_2)\in\cR_1\cup\cR_3$ it is well-known that if $T_+(\beta_1,\beta_2)/T_-(\beta_1,\beta_2)$ is an integer then both branches may actually correspond to the same periodic solutions of {\bf [2d]} -just traveled through several times on each period. Even when the nonresonance condition $T_+(\beta_1,\beta_2)/T_-(\beta_1,\beta_2)\not\in\mathbb Z$ holds we cannot exclude the possibility that the long and short period branches are connected, and so end up being the same branch. There is some additional local information at hand concerning these branches. For instance, when $(\beta_1,\beta_2)\notin\bar\cR_0\cup C$ the `short period' closed orbits near $(T_-(\beta_1,\beta_2),q_0)$ have minimal period close to $T_-(\beta_1,\beta_2)$. Similarly, if $(\beta_1,\beta_2)\in\cR_1\cup\cR_3$ and $T_+(\beta_1,\beta_2)/T_-(\beta_1,\beta_2)\not\in\mathbb Z$, the `long period' closed orbits near $(T_+(\beta_1,\beta_2),q_0)$ have minimal period close to $T_+(\beta_1,\beta_2)$. Moreover, the branch of closed orbits near the emanating point fills out a continuously embedded 2-dimensional disk, see e.g. \cite[\S 2.2]{MosZeh}.

\smallskip

 There are also some special situations where  the dynamics are less understood. Even though these cases are non-generic, they are often hard to exclude in particular problems.} For instance, the possibility $(\beta_1,\beta_2)\in(\partial\cR_0)\backslash C$ corresponds to the strongly resonant situation in which there are double purely imaginary characteristic exponents $\pm i\,2\pi/T_+(\beta_1,\beta_2)=$ $\pm i\,2\pi/T_-(\beta_1,\beta_2)\not=0$, whereas if $(\beta_1,\beta_2)\in C$ then the equilibrium is degenerate, i.e., $0$ is a characteristic exponent.  Thus, when $(\beta_1,\beta_2)\in C\backslash(\partial\cR_0)$ the usual versions of Lyapunov's center theorem do not apply, but Theorem \ref{thm:main2d}{\em (ii)} above states the existence of an emanating branch of closed orbits provided only that the Brouwer index $\ib(q_0,V')$ does not vanish. We emphasize that, in general, well-known examples of pathologies may occur when the assumptions of Lyapunov's center theorem fail. See, e.g., examples 9.1 and 9.2 in \cite[p. 210]{MawWill}.

\smallskip

Let us now turn our attention to the spatial case { \bf[3d]}. It seems reasonable that the additional dimension makes more space for the existence of emanating branches of closed orbits. As before, we denote by $\beta_1, \beta_2$ the eigenvalues of $\widetilde V''(\widetilde q_0)$ and set $\beta_3:=\beta_3(q_0)=\frac{\partial^2 V}{\partial z^2}(q_0)>0$. One has
\begin{theorem}\label{thm:main3d}
Assume {\bf [3d]}. If  $(\beta_1, \beta_2)\notin C$ then there is a branch of closed orbits emanating from $\left(\frac{2\pi}{\sqrt{\beta_3}},q_0\right)$. The same conclusion holds when  $(\beta_1, \beta_2)\in C$ provided that  $\ib(\widetilde q_0,\widetilde V') \not = 0$.
 \end{theorem}
Remembering assertion {\em (i)} of Theorem \ref{thm:main2d} we see that if $(\beta_1,\beta_2)\in\cR_0$ the emanating branch must be nonplanar. We emphasize that when $(\beta_1,\beta_2)\in C$ the critical point $q_0$ is degenerate and Lyapunov's center theorem does not apply.  The same arguments that lead us to Corollary \ref{corol1XT} give rise now to a result without direct assumptions on the spectrum $\beta_1,\beta_2$ of $V''(q_0)$.
\begin{corollary}\label{corUX}
	{Assume  { \bf[3d]}. If either $\widetilde V$ attains a local extremum at $q_0$, or $\widetilde V$ is even with respect to $q_0$, then there is a branch of closed orbits of \eqref{soe} emanating from $q_0$.}
\end{corollary}
We point out that Theorems \ref{thm:main2d}-\ref{thm:main3d} extend \cite[Proposition 3]{GarIze}, which is concerned only with a particular family of potentials $V$ and leaves aside the degenerate cases. Before closing this section we propose a pathological example for the planar case showing that bifurcation may not occur when some assumptions of Theorem \ref{thm:main2d} fail.
  	Consider the system 
 	\begin{equation*}
 		\begin{cases}
 			\ddot x+2\dot y-x-x^3-xy^2=0,\\
 			\ddot y-2\dot x-y-y^3-x^2y=0,
 		\end{cases} 
 	\end{equation*} 
 	on the domain $\Omega = \bR^2$, with the only equilibrium $q_0=(0,0)$.
 	Observe that it has the form {\bf [2d]} with $V(x,y)=-(x^2+y^2)/2-(x^2+y^2)^2/4$. In this case $(\beta_1,\beta_2)=(-1,-1)\in\partial\cR_0$ and Theorem \ref{thm:main2d} gives no information. It turns out that the only periodic solution of this system is the equilibrium $q_0=(0,0)$. In fact, every solution $(x,y) \not\equiv (0,0)$ satisfies
 	$$\frac{d^2}{dt^2}\left(\frac{x^2+y^2}{2}\right)=x\ddot x+y\ddot y+\dot x^2+\dot y^2=(\dot x+y)^2+(\dot y-x)^2+(x^2+y^2)^2>0,$$
 	and thus, it cannot be periodic.

 	\smallskip
 
 In view of this example, an open question appears: in the planar case {\bf [2d]}, and assuming that $(\beta_1,\beta_2)\in(\partial\cR_0)\backslash C$, is it possible to find sufficient conditions either implying or ruling out the existence of emanating branches of closed orbits? In case  $(\beta_1,\beta_2)\in((\partial\cR_0)\cap C)\backslash\{(0,-4),(-4,0)\}$ equality \eqref{bd} leads us to conjecture that a bifurcating branch of closed orbits does exist if $\ib(q_0,V')<0$ and does not exist if $\ib(q_0,V')>0$; however, this is not proven in this paper.
 For the spatial case { \bf[3d]} one can adapt the example above to show that there are situations  where $(\beta_1,\beta_2)\in(\partial \cR_0)\backslash C$ and the bifurcating branch of closed orbits, whose existence is granted by Theorem \ref{thm:main3d}, is nonplanar. In the case  where $(\beta_1, \beta_2)\in C\cap(\partial\cR_0)$ and  $\ib(\widetilde q_0,\widetilde V') \not = 0$ we do not know how to either guarantee or rule out that the branch of closed orbits emanating from $\left(\frac{2\pi}{\sqrt{\beta_3}},q_0\right)$ is planar. On the other hand, we have no idea on whether the condition on the Brouwer index  $\ib(\widetilde q_0,\widetilde V')$ (which coincides with $\ib(q_0,V')$, see Lemma \ref{lem2}) can be removed from Theorem \ref{thm:main3d}.

\section{The restricted triangular 4-body problem}\label{sec:celmec}

In this section we apply the general results above to study the motion of a massless test particle subjected to the gravitational attraction of three primaries of (possibly different) masses $m_1,m_2,m_3>0$, which occupy the vertices of a Lagrangian equilateral triangle rotating solidly around their center of masses at constant angular speed.  

\smallskip

The literature concerning this problem goes back to the dawn of the twentieth century. In \cite[\S 8]{Mou} one already reads that if the three primaries have equal mass then the number of libration points must be 10. After some papers in the forties \cite{Hin,Ped}, the development of computers allowed the introduction of numerical methods at the end of the seventies and beginning of the eighties \cite[\S 3]{Are}, \cite{Gan}, \cite[\S 2]{Sim}. For instance, in \cite[p. 168]{Sim} we find that for some choices of the masses the number of libration points can be just 8, while in \cite[p. 14]{Are} it is announced that the number of libration points can be 8, 9 or 10 depending on the masses. Some geometrical insight in the nineties  \cite[\S III]{MeySch2} has been  followed by a number of computer-assisted proofs in the new millennium \cite{BarLea, KulRobSmi, Lea}, which have given a renewed interest to the problem.

\smallskip

After changing units  in mass, space, and time, there is no loss of generality in assuming that the angular speed of the primaries is $1$, the side of the equilateral triangle is $\sqrt{3}$, and $m_1+m_2+m_3=3\sqrt{3}$. Then, the gravitational constant must be $G=1$, see \cite[\S 2.8]{Pol}. After choosing a synodic frame of reference we may assume that the primaries are fixed at the three cubic roots of unity:  $q_1=(1,0),\ q_2=(-1/2,\sqrt{3}/2),\  q_3=(-1/2,-\sqrt{3}/2)$. We are led to the effective potential 
\begin{equation}\label{VV}
	V(q):=-\frac{|q-c|^2}{2}-\frac{m_1}{|q-q_1|}-\frac{m_2}{|q-q_2|}-\frac{m_3}{|q-q_3|},\qquad q\in\Omega:=\mathbb R^2\backslash\{q_1,q_2,q_3\},
\end{equation}
where $c:=\frac{1}{3\sqrt 3}{(m_1q_1+m_2 q_2+m_3 q_3)}$ is the center of masses. It does not depend on time but, the three masses being possibly different, it may not coincide with the origin. 

\smallskip

In the plane $\bR^2$ we draw three circles of radius $\sqrt 3$ around the positions of each of the primaries and extend in both directions the three  sides of the triangle until they meet these circles again. In this way we obtain a compact set $\cS$ made up of three segments of common length $3\sqrt{3}$ and three circles of common radius $\sqrt{3}$. It divides the plane into seventeen open connected components, but we shall be interested in the open solid triangle $\cT$, the three open circular sectors $\cD_1,\cD_2,\cD_3$, and the three open circular triangles $\cO_1,\cO_2,\cO_3$. See Fig. 2, a picture which is nowadays classical. To the best of our knowledge it appeared firstly in \cite[p. 48]{Ped}, later in \cite[Fig. 2(d), p. 168]{Sim}, and more recently in \cite[p. 330]{Lea} and \cite[p. 1195]{BarLea}. The main result of this section is the following:

\begin{figure}
	\begin{center}
		\begin{tikzpicture}[xscale=1.2, yscale=1.2]
			\coordinate (A) at (1,0) ;
			\coordinate (B) at (-0.5,0.866);
			\coordinate (C) at (-0.5,-0.866);
			\draw[lightgray,->, thin](-3,0)--(3.5,0);
			\draw[lightgray,->, thin](0,-3.4)--(0,3.4);
			\fill[white!100] (-0.6,0) circle (0.2);	
			\fill[white!100] (0,0) circle (0.6);
			\draw[olive,densely dotted,very thick](A)--(B)--(C)--(A);
			\draw[gray,densely dotted,very thick](C)--(-0.5,-2.6);
			\draw[gray,densely dotted,very thick](C)--(-2,-1.733);
			\draw[gray,densely dotted,very thick](B)--(-0.5,2.6);
			\draw[gray,densely dotted,very thick](B)--(-2,1.733);
			\draw[gray,densely dotted,very thick](A)--(2.5,0.866);	
			\draw[gray,densely dotted,very thick](A)--(2.5,-0.866);	
			\draw[gray,densely dotted,very thick](A)--(2.5,-0.866);	
			
			\draw[gray,loosely dashed](2.2, 1.0388)--(1.3, 1.5572);	
			\draw[gray,loosely dashed](2, 0.9)--(1.28, 1.31456);	
			\draw[gray,loosely dashed](1.7, 0.7)--(1.25, 0.9591);	
			\draw[gray,loosely dashed](1.5, 0.5)--(1.23, 0.65546);
			
			\draw[gray,loosely dashed](2.2, -1.0388)--(1.3, -1.5572);	
			\draw[gray,loosely dashed](2, -0.9)--(1.28, -1.31456);	
			\draw[gray,loosely dashed](1.7, -0.7)--(1.25, -0.9591);	
			\draw[gray,loosely dashed](1.5, -0.5)--(1.23, -0.65546);

			\draw[gray,loosely dashed](-0.200373, 2.42466)--(0.698575, 1.90443);	
			\draw[gray,loosely dashed](-0.220577, 2.18205)--(0.498442, 1.76579);	
			\draw[gray,loosely dashed](-0.243782, 1.82224)--(0.205605, 1.56208);	
			\draw[gray,loosely dashed](-0.316987, 1.54904)--(-0.047355, 1.39294);

			\draw[gray,loosely dashed](-0.200373, -2.42466)--(0.698575, -1.90443);	
			\draw[gray,loosely dashed](-0.220577, -2.18205)--(0.498442, -1.76579);	
			\draw[gray,loosely dashed](-0.243782, -1.82224)--(0.205605, -1.56208);	
			\draw[gray,loosely dashed](-0.316987, -1.54904)--(-0.047355, -1.39294);
			
			\draw[gray,loosely dashed](-2,0.4)--(-2,1.4);	
			\draw[gray,loosely dashed](-1.8, 0.514)--(-1.8,1.35);	
			\draw[gray,loosely dashed](-1.5, 0.75)--(-1.5, 1.2);	
			\draw[gray,loosely dashed](-1.2, 0.85)--(-1.2, 1.1);
			
			\draw[gray,loosely dashed](-2,-0.4)--(-2,-1.4);	
			\draw[gray,loosely dashed](-1.8, -0.514)--(-1.8,-1.35);	
			\draw[gray,loosely dashed](-1.5, -0.75)--(-1.5, -1.2);	
			\draw[gray,loosely dashed](-1.2, -0.85)--(-1.2, -1.1);

			\draw[orange,densely dotted, very thick] (A) circle (1.73);  
			\draw[teal,densely dotted, very thick] (B) circle (1.73);  
			\draw[cyan, densely dotted, very thick] (C) circle (1.73);  
			\fill[white!100] (1.35,0) circle (0.16);
			\fill[violet!100] (A) circle (0.08); 
			\node at (1.4,-0.02)  {\color{violet} \scriptsize $q_1$};
			\fill[red!100] (B) circle (0.08);
			\node at (-0.7,1.2)  {\color{red} \scriptsize $q_2$};
			\node at (-0.7,-1.2)  {\color{brown} \scriptsize $q_3$};
			\fill[brown!100] (C) circle (0.08);
			\fill[white!100] (2.2,0) circle (0.3);
			\node at(2.18,-0.03){\color{black}$\cD_1$};
			\node at(-1.1, 1.8){\color{black}$\cD_2$};
			\node at(-1.1, -1.8){\color{black}$\cD_3$};
			\fill[white!100] (-1.2,0) circle (0.3);
			\node at(-1.2, -0.03){\color{black}$\cO_1$};
			\node at(0.6, -1){\color{black}$\cO_2$};
			\node at(0.6, 1){\color{black}$\cO_3$};
			\node[olive] at(0,-0.03){$\cT$};
			
			\draw[gray,dashed](-0.65,0.3)--(-0.5,0.3);	
			\draw[gray,dashed](-0.65,-0.3)--(-0.5,-0.3);	\draw[gray,dashed](-0.65,0)--(-0.5,0);	
			
			\draw[gray,dashed](0.0651924, -0.712917)--(0,-0.5848);	
			\draw[gray,dashed](0.584808, -0.412917)--(0.509808, -0.283013);	\draw[gray,dashed](0.325, -0.562917)--(0.25, -0.433013);

			\draw[gray,dashed](0.0651924, 0.712917)--(0,0.5848);	
			\draw[gray,dashed](0.584808, 0.412917)--(0.509808, 0.283013);	\draw[gray,dashed](0.325, 0.562917)--(0.25, 0.433013);	
			
		\end{tikzpicture}
	\end{center}
	\caption{$\cS$ is pictured with dotted lines. Notice that $\R^2\backslash\cS$ has $17$ open connected components, including the unbounded one. }
\end{figure}
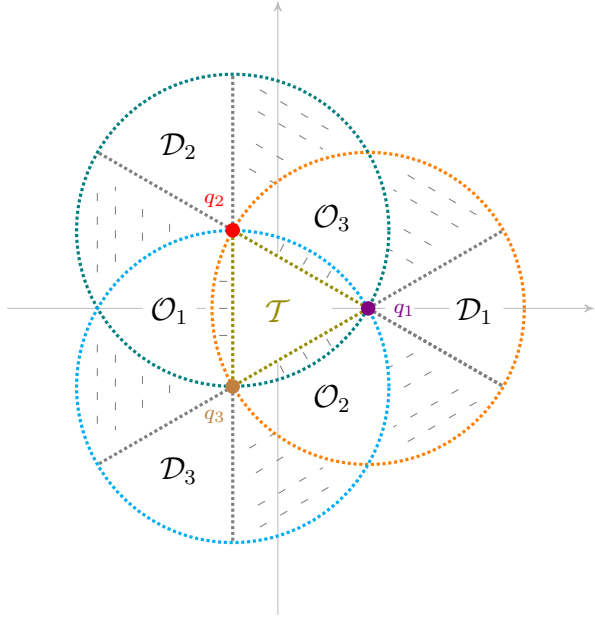
\begin{theorem}\label{th:libration}
	{For any choices of the masses $m_1,m_2,m_3>0$ with $m_1+m_2+m_3=3\sqrt 3$ the following hold:
		\begin{enumerate}
			\item[(i)] There are at least $7$ libration points for this problem: at least one in $\cT$, at least one in each $\cO_i$ and at least one in each $\cD_i$.
			\item[(ii)]  After identifying the plane $\bR^2$ with the horizontal plane $\bR^2\times\{0\}\subset\bR^3$ and regarding the potential $V$ in \eqref{VV} as defined on $\bR^3\backslash\{q_1,q_2,q_3\}$, these seven planar libration points can be chosen so that there are (either planar or spatial) branches of closed orbits emanating from all of them. 
		\end{enumerate}}
\end{theorem}

We emphasize that this result does not require nondegeneracy assumptions on the libration points. In particular, the classical versions of Lyapunov's center theorem do not seem directly applicable.

\smallskip

According to  \cite[p. 14]{Are} or \cite[p. 1186]{BarLea}, the minimal number of libration points in this problem is 8, and thus, our multiplicity result is not optimal; on the other hand we present a new proof which is based on Brouwer's degree theory and does not rely on computers.

\smallskip

The lemma below is well-known. 	To the best of our knowledge, the statement in {\em (a)} was first proved in \cite[II. 7]{Ped}, see also \cite[Lemma 3.2]{Lea}.  The statement in {\em (b)} can be obtained by combining the results in \cite[\S 4]{Hin} and, for instance, \cite[Proposition 6{\em (b)}]{GarIze}. 
\begin{lemma}\label{coro3217}
	{The following hold:
		\begin{enumerate}
			\item[(a)]For any values of the masses $m_1,m_2,m_3>0$ with $m_1+m_2+m_3=3\sqrt 3$ one has $$(V')^{-1}(0)\subset\cT\cup\bigcup_{i=1}^3\cD_i\cup\bigcup_{i=1}^3\cO_i.$$
			\item[(b)] In the particular case $m_1=m_2=m_3=\sqrt 3$ there are exactly ten critical points of $V$, all of which are nondegenerate. More precisely, there is a local maximum at the origin, three additional saddle points in $\cT$, one local maximum in each set $\cO_i$, and one saddle point in each set $\cD_i$.	
	\end{enumerate}}
\end{lemma}

In particular, Lemma \ref{coro3217}{\em (a)} states that $V$ does not have critical points on $\cS$. Since $|V'|$ can be considered a continuous map from $\bR^2$ to $(0,+\infty]$, a standard compactness argument shows that the minimal distance from $\cS$ to the set of critical points of $V$ is bounded from below by some positive constant. It allows us to consider the {\em generalized Brouwer degrees} 
$$\db(V',\cT),\qquad \db(V',\cO_i),\qquad \db(V',\cD_i),\qquad i=1,2,3,$$
by which we mean the Brouwer degrees of $V'$ on slightly smaller open sets. For instance, setting $\cT_\epsilon:=\{q\in\cT:\dist(q,\partial\cT)>\epsilon\}$, we define $\db(V',\cT)$ as the Brouwer degree of $V'$ on $\cT_\epsilon$ for $\epsilon>0$ small enough. 

\smallskip

We observe that, even though the potential $V$ depends on the masses $m_1,m_2,m_3$, the above Brouwer indexes do not. To check this assertion we first point out that, if the masses are assumed to belong to some compact subset of $\{(m_1,m_2,m_3)\in\mathbb R^3:m_i>0\ \forall i=1,2,3,\ m_1+m_2+m_3=3\sqrt 3\}$, then the positive lower bound for the distance between points of $\cS$ and critical points of $V$ can be chosen uniformly with respect to the masses. Therefore our claim comes from the homotopy invariance of the Brouwer degree. We arrive to the following

\begin{lemma}\label{lem123}{$\db(V',\cT)=-2$,  $\db(V',\cO_i)=1$ and $\db(V',\cD_i)=-1$ for $i=1,2,3$. It holds true for any choice of the masses $m_1,m_2,m_3>0$ with $m_1+m_2+m_3=3\sqrt 3$.}
	\begin{proof}In view of the comments above, it suffices to show the result when $m_1=m_2=m_3=\sqrt 3$. In this case the result follows immediately from Lemma \ref{coro3217}{\em (b)} and the additivity property of the Brouwer degree.
	\end{proof}	
\end{lemma}
\begin{proof}[Proof of Theorem \ref{th:libration}] Statement {\em (i)} follows directly Lemma \ref{lem123}. On the other hand, {\em (ii)} now follows from Theorem \ref{thm:main3d}. In fact, assumptions {\bf [H$_{1-2}$]} are easy to check, while assumption {\bf [V]} was proven for this situation in \cite[Theorem 2.1]{KulRobSmi}. It concludes the argument.
\end{proof}

\section{Branches of closed orbits emanating from equilibria of Hamiltonian systems}\label{sec:branches}

Many results concerning the bifurcation of closed orbits from equilibria of Hamiltonian systems are available in the literature. In this section we present, without proof, a classic theorem in this direction that can be obtained by means of degree theory for equivariant gradient maps. We also recall De Gua's corollary of Descartes' rule of signs, a result which allows the computation of  the total multiplicity of the positive roots of a real polynomial by looking at the number of sign changes in the list of its coefficients.  The sole assumption is that the polynomial  should have {\em only real roots}, which is the case for the characteristic polynomial of a symmetric matrix.

\smallskip 

Consider a general (autonomous) Hamiltonian system:
\begin{equation*} \label{hs} \tag{HS}
	\dot u(t)= J_N H'(u(t))\,. 
\end{equation*}
Here  $J_N=\left(\begin{array}{c|c}
	0_N &- I_N \\ \hline I_N & 0_N\end{array}\right) \in\mathbb R^{2N\times 2N}$ is the standard symplectic   matrix,  $H:\cU\to\mathbb R$ is a $C^2$ function defined on the  open set $\cU\subset\R^{N}\times\R^N$, and $H'$ stands for its gradient. Through this section $N$ could be any natural number, and is not restricted to $2$ or $3$. We further assume that all critical points of $H$ are isolated. We pick an stationary point $u_0\in\cU$ and denote by $A:=H''(u_0)$ the Hessian matrix of $H$ at $u_0$.

\smallskip

It is well-known that the eigenvalues of the Hamiltonian matrix $J_NA$ (usually referred to as the characteristic exponents at $u_0$) play an important role in the dynamics of (\ref{hs}) near the stationary solution. For instance, the Hartman-Grobman theorem implies that if
$\sigma( J_NA)\cap(i\mathbb R)=\emptyset$
then there are no closed orbits of \eqref{hs} in a neighborhood of $u_0$. This is the so-called hyperbolic case.

\smallskip

On the other hand,  a sufficient condition is given by  Lyapunov's center theorem: under the nondegeneracy condition  $\det A\not=0$, the presence of a purely-imaginary characteristic exponent $\lambda i$ with algebraic multiplicity $1$ and no higher-order resonances implies the existence of an emanating local branch of closed orbits.

\smallskip

We are interested in generalizing this result in two directions: firstly, we would like to soften the condition on the purely-imaginary eigenvalue to be simple, and secondly we wish to obtain global (in the sense of Definition \ref{branch}) rather than local branches. With this purpose we consider, for each $T>0$,  the symmetric matrix
\begin{equation} \label{tkla} 	\renewcommand*{\arraystretch}{1.5} S_T:=\left(\begin{array}{c|c}
		-\frac{T}{2\pi}A& - J_N\\ \hline J_N&-\frac{T}{2\pi}A
	\end{array}\right)\in\R^{4N\times 4N}\,.
\end{equation}

To the best of our knowledge, this kind of construction was introduced by Szulkin \cite{Szu} while studying the local bifurcation of closed orbits via equivariant Morse theory; in fact, $S_T=T_1\left(\frac{T}{2\pi}A\right)$ in Szulkin's symbols. We have opted for this change in the notation since the letter $T$ stands for period throughout this paper. 

\smallskip

Calling $\Upsilon$ the $4N$-dimensional vector space

$$\Upsilon:=\{(\sin\theta)v_1+(\cos\theta)v_2:v_1,v_2\in\R^{2N}\}\subset C(\mathbb S^1,\mathbb R^{2N})\,,$$ 
denoting by $L_T:\Upsilon\to \Upsilon$ the linear map defined by
$$L_T\bar\zeta:=-J_N\left[\left(\frac{d\bar\zeta}{d\theta}\right)-\frac{T}{2\pi}J_NA\bar\zeta\right]\,,\qquad \zeta\in \Upsilon\,,$$
and letting $\{e_1,\ldots,e_{2N}\}$ be the canonical basis of $\mathbb R^{2N}$,
one can think of $S_T$ as  being the matrix of $L_T$ with respect to the basis of $\Upsilon$ given by $$\big\{(\sin\theta)e_1,(\sin\theta)e_2,\ldots,(\sin\theta)e_{2N}, (\cos\theta)e_1,(\cos\theta)e_2,\ldots,(\cos\theta)e_{2N}\big\}.$$ Therefore, via the linear reparametrization $\theta=2\pi t/T$, the kernel of
$S_T$ corresponds to the set of sinusoidal closed curves of pure frequency $2\pi/T$ which solve the linearization of (HS) at $u(t)\equiv u_0$.

\smallskip

Given $T>0$ it is well-known that $S_T$ is singular if and only if $\frac{2\pi}{T} i\in\sigma(J_N A)$; we include a direct proof in Corollary  \ref{corolnew} for completeness. Letting $T$ vary on $(0,+\infty)$, the  Morse index $\morse(S_T)$ may change only at these values.  For any $T>0$ the bifurcation number\footnote{The so-called {\em bifurcation index} has been more extensively studied in the literature, see, e.g., \cite{DanRyb, GolRyb}. The bifurcation number considered here is just the $\mathbb Z_1$-component of the  bifurcation index.} 
$\gamma_H(T,u_0)$ is defined as follows: 
\begin{equation}\label{bb3}\gamma_H(T,u_0):=\ib(u_0,H')\lim_{\epsilon\searrow 0}\left(\frac{\morse(S_{T+\epsilon}) - \morse(S_{T-\epsilon})}{2} \right).
	\end{equation}
For instance, in the nondegenerate case $\det(H''(u_0))\not=0$ and 
$\gamma_H(T,u_0)$ is the sign of $\det(H''(u_0))$ times $(\morse(S_{T+\epsilon}) - \morse(S_{T-\epsilon}))/2$ for $\epsilon>0$ small enough. We point out that the bifurcation number is an integer because the Morse index $\morse(S_T)$, which coincides with the total multiplicity of the negative eigenvalues of $S_T$, is always even, see Corollary \ref{cor221}.  On the other hand, $\gamma_H(T,u_0)$ will certainly be zero if $\frac{2\pi}{T}i$ is not a characteristic exponent. 
The result from equivariant degree theory which we shall need in this paper is the following:
\begin{theorem}[Dancer and Rybicki \cite{DanRyb}]\label{th:bifurcation}{If $\gamma_H(T,u_0)\not=0$ for some $T>0$, then there is a branch $\widehat{\cB}$ of closed orbits of (HS) emanating from $(T,u_0)$.}
\end{theorem}
Throughout this paper and also in this result, the word {\em branch} should be understood in the sense of Definition \ref{branch}. Thus, by a {\em branch} of solutions of (HS) we mean a connected component $\widehat{\mathcal B}$ of the closure $\widehat\Lambda$ of the set of nontrivial closed orbits of (HS) which, either: {\em (a)} is unbounded, or {\em (b)} goes up to the boundary of $\mathcal U$, or {\em (c)} is compact  and contains at least two trivial closed orbits. Moreover, in this latter case the sum of the bifurcation numbers of the trivial closed orbits in $\widehat{\cB}$ can be shown to vanish, an observation which will not be used in this paper.

\smallskip

We point out that, while the original literature deals with globally-defined Hamiltonians $H:\bR^N\times\bR^N\to\bR$ (thus excluding  possibility {\em (b)}), the more general situation considered in this paper can be dealt with by using the same arguments. 

\smallskip

In the particular case of Hamiltonian systems, Theorem \ref{th:bifurcation} extends the classical statement of Lyapunov's center theorem: it can be checked that if $\lambda i$ with $\lambda>0$ is an algebraically simple characteristic exponent then $\morse(S_T)$ changes when $T$ crosses $2\pi/\lambda$. In fact, this statement keeps its validity if  $\lambda i$ has odd algebraic multiplicity, giving rise to results in the line of Krasnoselskii's celebrated theorem. We shall not use these facts in this paper.

\smallskip

In order to apply Theorem \ref{th:bifurcation} one needs to compute the Morse indexes of the matrices $S_T$. Since these matrices are real and symmetric, their eigenvalues are real and De Gua's corollary of Descartes' rule of signs \cite[Th\'eor\`eme III]{Gua} will be useful. The precise statement of this result is given next: 
\begin{theorem}
	[De Gua \cite{Gua}]\label{DeG}{Let $p(\lambda)=d_k\lambda^k+d_{k-1}\lambda^{k-1}+\ldots+d_1\lambda+d_0$ be a real polynomial without complex nonreal roots. Then the total multiplicity of the positive roots of $p$ coincides with the number of sign changes in the ordered list of coefficients $$d_k,d_{k-1},\ldots,d_1,d_0,$$ where the zero elements that might possibly occur are to be removed. 
	}
\end{theorem}

\section{From a second-order equation to a Hamiltonian system}\label{sec:hamiltonian}
From now on our goal will be to prove the results announced in Section \ref{sec:results}. For this reason, and until the end of Section \ref{sec:prf}, we go back now to the general framework and notation of Sections \ref{sec:def}-\ref{sec:results}. We shall start with the following observation: setting $p:=\dot q-\alpha_N q$ and  $u=(p,q)\in\mathbb R^N\times\mathbb R^N$, equation \eqref{soe} can be rewritten as a Hamiltonian system. More precisely, one  gets system (HS) for the Hamiltonian function
\begin{equation}\label{H}
	H(p,q)=\frac{1}{2}|p|^2+\langle p,\alpha_Nq\rangle+W(q),\qquad (p,q)\in\cU:=\bR^N\times\Omega,
\end{equation}  
the {\em amended potential} $W:\Omega\to\mathbb R$ being defined by $$W(q):=V(q)-\frac{1}{2}\langle q,\alpha_N^2 q\rangle,\qquad q\in\Omega,$$
or, more explicitly,
$$\begin{cases}
	W(x,y)=V(x,y)+(x^2+y^2)/2,\qquad\qquad\  (x,y)\in\Omega,\qquad\ &\text{in case }{\bf [2d]},\\
	W(x,y,z)=V(x,y,z)+(x^2+y^2)/2,\qquad\  (x,y,z)\in\Omega,\qquad &\text{in case }{\bf [3d]}.
\end{cases}$$
From now on {\em  it will always be assumed that the Hamiltonian $H$ is given by \eqref{H}}.

\smallskip

The equilibria of (HS) are in a 1:1 correspondence with the equilibria of (\ref{soe}). More specifically, $u_0=(p_0,q_0)$ is a critical point of $H$ if and only if $q_0$ is a critical point of $V$ and $p_0=-\alpha_Nq_0$. In this case, for any $T>0$ we shall say that the pair $(T,u_0)$ is a {\em trivial closed orbit} of (HS).

\smallskip

On the other hand, the map $\Phi(T,\bar q):=(T;\bar q,\bar p)$ defined by $\bar p:=(2\pi/T)(d\bar q/d\theta)-\alpha_N\bar q$ establishes a 1:1 correspondence between the closure $\Lambda$ of the set of nontrivial closed orbits of (\ref{soe}), which we see as a subset of $(0,+\infty)\times C(\mathbb S^1,\Omega)$, and the closure $\widehat\Lambda$ of the set of nontrivial closed orbits of (HS), regarded as a subset of  $(0,+\infty)\times C(\mathbb S^1,\cU)$. 
There is no difficulty in translating Definition \ref{branch} to this context: by a {\em branch} of solutions of (HS) we shall mean a connected component $\widehat{\mathcal B}$ of $\widehat\Lambda$ which, either: (a) is unbounded, or (b) goes up to the boundary of $\mathcal U$, or (c) is compact and contains at least two trivial closed orbits.

\smallskip

There are a number of properties that $q_0$ may have as an equilibrium of \eqref{soe} and are inherited by $u_0$ as an equilibrium of (HS); we collect some of them in Lemma \ref{lem1} below. Assertion {\em (ii)} below should be read in the spirit of the comments  following the statement of Theorem \ref{thm:main2d}. 
\begin{lemma}\label{lem1}{Let $u_0=(p_0,q_0)$ be an equilibrium of (HS). Then, the following hold:
		\begin{enumerate}
			\item[(i)] $u_0$ is isolated as a critical point of $H$.
			\item[(ii)] System \eqref{soe} does not have closed orbits in a sufficiently small neighborhood of $q_0$ if and only if (HS) does not have closed orbits in a sufficiently small neighborhood of $u_0$.
			\item[(iii)] Given $T>0$, there is a branch of closed orbits of \eqref{soe} emanating from $(T,q_0)$ if and only if there is a branch of closed orbits of (HS) emanating from $(T,u_0)$.
\end{enumerate}}\end{lemma}

\begin{proof} {\em (i)}: 	The map $q\in\Omega\mapsto(-\alpha_N q,q)$ is an embedding sending $q_0$ into $u_0$, and, both in the planar and the spatial cases, it carries the critical points of $V$ in a neighborhood of $q_0$ into the critical points of $H$ in a neighborhood of $u_0$. Therefore, assumption  {\bf [V]} implies the statement. {\em (ii)}: The nontrivial implication is contained in the following claim: {\em if $q_n:\mathbb R\to\mathbb R^N$ is a sequence of solutions of \eqref{soe} uniformly converging to $q_0$ then $\dot q_n\to 0$ uniformly on $\mathbb R$}. We check this statement by a contradiction argument and assume that  $\{q_n\}$ is as above but there exists some sequence $\{t_n\}\subset\mathbb R$ such that $|\dot q_n(t_n)|\geq\epsilon$ for some $\epsilon>0$ and every $n\in\mathbb N$. Since equation \eqref{soe} is autonomous there is no loss of generality in assuming that $t_n=0$ for every $n$. On the other hand, after possibly passing to a subsequence there is no loss of generality in assuming that either $\dot q_n(0)\to\dot q_0\in\mathbb R^N$ or $|\dot q_n(0)|\to+\infty$ as $n\to+\infty$. In the first case, continuous dependence would imply that $q_n$ converges uniformly on the compact interval $[0,1]$ to the solution of $\ddot q-2\alpha_N\dot q+V'(q)=0,\ q(0)=q_0,\ \dot q(0)=\dot q_0$, which contradicts the fact that $q_n\to q_0$ uniformly on $\mathbb R$. In the second case one can repeat the argument with $r_n:=(1/|\dot q_n(0)|)q_n$, which, for each $n\in\mathbb N$ solves the linear equation $\ddot r-2\alpha_N\dot r+V'(q_n(t))/|\dot q_n(0)|=0$, and passing to the limit as $n\to+\infty$ one arrives similarly to a contradiction. {\em (iii)}: One immediately checks that a set $\mathcal B\subset\Lambda$ is a branch of closed orbits of (\ref{soe}) if and only if $\widehat{\cB}:=\Phi(\cB)$ is a branch of closed orbits of (HS).
\end{proof}

\smallskip

We close this section with a result relating the Brouwer indexes of $H'$ and $V'$ at the isolated equilibria $u_0=(p_0,q_0)$ and $q_0$, respectively. 

\begin{lemma}\label{lem2}
	{Let $u_0=(p_0,q_0)\in\mathbb R^N\times\Omega$ be a critical point of $H$. Then,  $\ib(u_0, H')=\ib(q_0, V')$. Moreover, in the spatial case   this Brouwer index coincides with $\ib(\widetilde q_0,\widetilde V')$.}  
	\begin{proof}
 We consider first the planar case  {\bf [2d]}. Setting $\Pi:=\left(\begin{array}{c|c}I_2&0_2\\ \hline \alpha_2& I_2\end{array}\right)$, which is a $4\times4$ matrix with determinant $1$, we see that
			$$\Pi H'(p,q)
			=(p+\alpha_2 q, V'(q))=\mathcal H_0(p, q)\,,\qquad (p,q)\in\mathbb R^2\times\Omega,$$
			where $\cH_\lambda(p,q):=(p+(1-\lambda)\alpha_2 q, V'(q))$. Now, writting $u_\lambda:=((1-\lambda)p_0, q_0)$ we notice that $\cH_\lambda(u_\lambda)=(0, 0)$ for every $\lambda\in[0,1]$. Moreover, $q_0$ being an isolated critical point of $V$, one easily checks that $u_\lambda$ is an isolated zero of $\cH_\lambda$ and indeed the Brouwer index $\ib(u_\lambda,\cH_\lambda)$ does not depend on $\lambda\in[0,1]$. Thus, $\ib((p_0,q_0),H')=\ib((p_0,q_0),\cH_0)=\ib((0,q_0),\cH_1)$. Since $\cH_1(p,q)=(p,V'(q))$, the multiplicative property of the Brouwer degree gives $\ib((0,q_0),\cH_1)=\ib(q_0,V')$ and thus concludes the proof.
			
			\smallskip
			
			In the spatial case {\bf [3d]} we use the notation $(p,q)\equiv(\widetilde p,\widetilde q,p_3,q_3)\in\mathbb R^4\times\mathbb R^2$ and observe that
		$$H'(p,q)=\left(\tilde H'(\widetilde p,\widetilde q);\ p_3,\frac{\partial V}{\partial z}(q)\right),$$
		where $\tilde H$ is the Hamiltonian corresponding to the planar case for the potential $\widetilde V$. Remembering {\bf [H$_{1-2}$]} and using the result for the planar case we see that
		$$\ib(u_0,H')=\ib((\widetilde p_0,\widetilde q_0),\tilde H')=\ib(\widetilde q_0,\widetilde V')=\ib(q_0,V')\,,$$
		proving the result.
	\end{proof}
\end{lemma}

\section{On the spectrum of the linearization} \label{sec:spectrum}

Almost any study of the dynamics of (HS) near the equilibrium $q_0$ should begin with an analysis of the spectrum of the Hamiltonian matrix $J_NH''(u_0)$. Notice that
\begin{equation}
	\label{A}A:=H''(u_0)=\left(\begin{array}{c|c}
		I_N&\alpha_N\\ \hline -\alpha_N&  W''(q_0)\end{array}\right),
\end{equation}
and we shall exploit this fact to reduce the order of certain determinants. The first result of this section, which is valid both in the planar and the spatial cases, shows that in the computation of the characteristic polynomial of $J_NA$ one can replace the $2N\times 2N$ matrix $ J_N A-\lambda I_{2N}$ by the $N\times N$ matrix $V''(q_0)+\lambda^2I_N-2\lambda\alpha_N$.

  \begin{lemma}\label{lem00}
  For any $\lambda \in \mathbb{C}$ the following equality holds:  $$\det(J_NA-\lambda I_{2N})=\det( V''(q_0)+\lambda^2I_N-2\lambda\alpha_N).$$
  \end{lemma}
  \begin{proof}
  	It is well-known (and follows from Lemma \ref{lem0}) that $\det J_N=1$. Moreover, one has:
  	\begin{multline*}
  	 J_N( J_NA-\lambda I_{2N})=-A-\lambda J_N=-\left(\begin{array}{c|c}
  	I_N&\alpha_N\\ \hline -\alpha_N& W''(q_0) \end{array}\right)-\lambda \left(\begin{array}{c|c}
  	0_N&-I_N\\ \hline I_N& 0_N\end{array}\right)=\\=\left(\begin{array}{c|c}
  	-I_N&-\alpha_N+\lambda I_N\\ \hline \alpha_N-\lambda I_N& -W''(q_0) \end{array} \right)\,.
  	\end{multline*}
  If, in the last term, one adds to the last $N$ rows the product of $\alpha_N-\lambda I_N$ by the first $N$ rows, one obtains $\renewcommand*{\arraystretch}{1.2}\left(\begin{array}{c|c}
  -I_N&-\alpha_N+\lambda I_N\\ \hline 0_N& -W''(q_0)-(\alpha_N-\lambda I_N)^2\end{array}\right)$; in particular, this latter matrix and $J_NA-\lambda I_{2N}$ have the same  determinant. The result follows after remembering that $W''(q_0)+\alpha_N^2=V''(q_0)$.
\end{proof}

As in Section \ref{sec:results}, we denote by $\beta_1,\beta_2 \in \R$ the eigenvalues of $ V''(q_0)$ in the planar case, or the eigenvalues of $\widetilde V''(\widetilde q_0)$ in the spatial problem; in the latter situation we also set $\beta_3:=\beta_3(q_0)=\frac{\partial^2 V}{\partial z^2}(q_0)>0$. Our next task is to compute explicitly the characteristic polynomial of $J_NA$ in terms of these numbers. 

  \begin{corollary}\label{corol}
  	{In the planar case the characteristic polynomial $\mathfrak p_2(\lambda):=\det( J_2A-\lambda I_{4})$ is given by
\begin{equation}
\label{p2}\mathfrak p_2(\lambda)=\lambda^4+(\beta_1+\beta_2+4)\lambda^2+\beta_1\beta_2\,,
\end{equation}
and in the three-dimensional case  {\bf [3d]} the characteristic polynomial $\mathfrak p_3(\lambda):=\det( J_3A-\lambda I_{6})$ is given by
\begin{equation*}
\mathfrak p_3(\lambda)=(\lambda^2+\beta_3)\,\mathfrak p_2(\lambda)\,,
\end{equation*}
the fourth-order polynomial $\mathfrak p_2(\lambda)$ being defined by (\ref{p2}).}
\end{corollary}
\begin{proof}
Let us start by considering the planar case {\bf [2d]}. Lemma \ref{lem00} gives
$$\mathfrak p_2(\lambda)=\det\big(V''(q_0)+\lambda^2I_2-2\lambda\alpha_2\big)\,.$$
On the other hand, $V''(q_0)\in\R^{2\times 2}$ is symmetric and therefore, after possibly changing the order of the $\beta_i$'s, there  exists a rotation $\mathscr R\in SO_2(\mathbb R)$ such that $\mathscr R^T V''(q_0)\mathscr R= \begin{pmatrix}
	\beta_1& 0\\
	0&\beta_2
	\end{pmatrix}$. The matrices $\mathscr R$ and $\alpha_2$ commute, and therefore,
	\begin{multline*}
	\mathfrak p_2(\lambda)=\det(\mathscr R^T V''(q_0)\mathscr R+\lambda^2I_2 -2\lambda\alpha_2)=\\=\det\begin{pmatrix}
	\beta_1+\lambda^2&2\lambda\\-2\lambda&\beta_2+\lambda^2
	\end{pmatrix}=\lambda^4+(\beta_1+\beta_2+4)\lambda^2+\beta_1\beta_2\,,
	\end{multline*}
thus showing (\ref{p2}).

\smallskip

In the three-dimensional situation one may differentiate in {\bf [H$_1$]} to find
\begin{equation*}
	\frac{\partial^2 V}{\partial x\partial z}(x,y,0)=0=\frac{\partial^2 V}{\partial y\partial z}(x,y,0),\qquad (x,y)\in\widetilde\Omega\,,
\end{equation*}
so that the Hessian of $V$ at the point $q_0=(\widetilde q_0,0)\in\widetilde\Omega\times\{0\}$ has the form
\[ V''(q_0)=
\left(
\begin{array}{c|c}
	\raisebox{-8pt}{\normalsize\mbox{{$\widetilde V''(q_0)$}}} & 0\\[-1.1ex]  & 0\\\hline\\[-2ex]  0\ \ \ \ \ \  0&\beta_3
\end{array}
\right)\,.
\]

The result follows from Lemma \ref{lem00}.
\end{proof}

 Set  $\mathfrak q(x):=x^2+(\beta_1+\beta_2+4)x+\beta_1\beta_2$, which is a quadratic polynomial with $\mathfrak p_2(\lambda)=\mathfrak q(\lambda^2)$. In this way there is a $1:1$ correspondence between the set of couples $\pm\lambda i\not=0$ of purely-imaginary roots  of $\mathfrak p_2$ and the set of real negative roots of $\mathfrak q$. We further observe that the discriminant of $\mathfrak q$ is given by $\Delta:=(\beta_1+\beta_2+4)^2-4\beta_1\beta_2=8(\beta_1+\beta_2+2+(\beta_1-\beta_2)^2/8)$. It follows that the nonzero, purely-imaginary roots of  $\mathfrak p_2$ are given by $\pm i\, 2\pi/T_-(\beta_1,\beta_2),$ $\pm i\, 2\pi/T_+(\beta_1,\beta_2)$, wherever defined. Here, the functions $T_-:(\bR^2\backslash\bar\cR_0)\cup((\partial\cR_0)\backslash C)\to(0,+\infty)$ and $T_+:\cR_1\cup\cR_3\cup((\partial\cR_0)\backslash C)\to(0,+\infty)$ are as described in \eqref{lpm}. Notice that:
\begin{lemma}\label{lem1234321}{$T_\pm$ are strictly positive in their respective domains. Moreover,	$T_-(\beta_1,\beta_2)\leq T_+(\beta_1,\beta_2)$ on $\cR_1\cup\cR_3\cup((\partial\cR_0)\backslash C)$, and the equality holds if and only if $(\beta_1,\beta_2)\in(\partial\cR_0)\backslash C$.
}
\end{lemma}

As announced in Section \ref{sec:results}, the functions $T_\pm$ can be used to draw a global picture of the purely-imaginary characteristic exponents at $u_0$. More precisely, one has the following:
\begin{lemma}\label{lem6}{In the planar case the following hold:
		\begin{enumerate} 
			\item[(i)] If $(\beta_1,\beta_2)\in \cR_0$ then $\sigma( J_2A)\cap(i\mathbb R)=\emptyset$.
			\item[(ii)] If $(\beta_1,\beta_2)\in\cR_2\cup \cR_4$ then $\sigma( J_2A)\cap(i\mathbb R)=\left\{-\frac{2\pi}{T_-(\beta_1,\beta_2)} i,\frac{2\pi}{T_-(\beta_1,\beta_2)} i\right\}$. 
			\item[(iii)] If $(\beta_1,\beta_2)\in\cR_1\cup\cR_3$ then $\sigma( J_2A)=\left\{-\frac{2\pi}{T_-(\beta_1,\beta_2)} i,\frac{2\pi}{T_-(\beta_1,\beta_2)} i,-\frac{2\pi}{T_+(\beta_1,\beta_2)} i,\frac{2\pi}{T_+(\beta_1,\beta_2)}\right\}$.
			\item[(iv)] If $(\beta_1,\beta_2)\in(\partial\cR_0)\backslash C$ then $\sigma( J_2A)=\left\{-\frac{2\pi}{T_-(\beta_1,\beta_2)} i,\frac{2\pi}{T_-(\beta_1,\beta_2)} i\right\}=\left\{-\frac{2\pi}{T_+(\beta_1,\beta_2)} i,\frac{2\pi}{T_+(\beta_1,\beta_2)} i\right\}$.
		\item[(v)] If $(\beta_1,\beta_2)\in C\cap(\partial\cR_0)$ then $\sigma( J_2A)\cap(i\mathbb R)=\{0\}$.
	\item[(vi)] If $(\beta_1,\beta_2)\in C\backslash(\partial\cR_0)$ then $\sigma( J_2A)=\left\{0,-\frac{2\pi}{T_-(\beta_1,\beta_2)} i,\frac{2\pi}{T_-(\beta_1,\beta_2)} i\right\}$.\end{enumerate}
	Furthermore, in the spatial situation one obtains $\sigma(J_3A)\cap(i\mathbb R)$ by adding $\pm\sqrt{\beta_3}i$ to the previously-described lists.
}
\begin{proof}There are several of cases to be considered, but the arguments being similar, we shall only study in detail the nondegenerate situation $\beta_1\not=0\not=\beta_2$ in which $\mathfrak p_2$ has exactly one pair $\pm\lambda i$ of purely imaginary roots.  This is equivalent to say that $\mathfrak q$ does not vanish at the origin and has exactly one negative rooot, which will happen if and only if, either $\mathfrak q(0)<0$, or $\mathfrak q(0),\dot{\mathfrak q}(0)>0$ and $\Delta=0$. The first possibility corresponds to case {\em (ii)} and then, the purely-imaginary roots $\pm\lambda i=\pm\frac{2\pi}{T_-(\beta_1,\beta_2)}i$ are simple. The second option corresponds to case {\em (iv)} and then, the purely-imaginary roots $\pm\lambda i=\pm\frac{2\pi}{T_-(\beta_1,\beta_2)}i=\pm\frac{2\pi}{T_+(\beta_1,\beta_2)}i$ have algebraic multiplicity 2.
\end{proof}
\end{lemma}

 \section{Computing Morse indexes via De Gua's corollary of Descartes' rule of signs} \label{sec:Tk}

The purpose of this section is to compute explicitly the Morse index of $S_T$ as a function of the parameter $T>0$ and the eigenvalues of $V''(q_0)$. In the planar situation  this latter plan will result in the following  

 \begin{proposition}\label{prop0}{In the planar case, the following hold for every $T>0$:
 		\begin{enumerate}
 			\item[(i)] If $(\beta_1,\beta_2)\in \cR_0\cup((\partial\cR_0)\cap C)$, then $\morse(S_T)=4$, independently of the value of $T>0$.
 		\item[(ii)] If $(\beta_1,\beta_2)\in(\partial\cR_0)\backslash C$, then $\morse(S_T)=4$ for any $T>0$ with $T\not=T_-(\beta_1,\beta_2)$.
 			\item[(iii)]If $(\beta_1,\beta_2)\in\cR_2\cup \cR_4\cup (C\backslash(\partial\cR_0))$, then $\morse(S_T)=\begin{cases}
 			4,&\text{if }0<T<T_-(\beta_1,\beta_2)\,,\\
 			6,&\text{if }T>T_-(\beta_1,\beta_2).
 		\end{cases}$
 		\item[(iv)]If $(\beta_1,\beta_2)\in  \cR_1$, then $\morse(S_T)=\begin{cases}
 		4,&\text{if }0<T<T_-(\beta_1,\beta_2)\,,\\
 		6,&\text{if }T_-(\beta_1,\beta_2)<T<T_+(\beta_1,\beta_2)\,,\\
 		8,&\text{if }T>T_+(\beta_1,\beta_2)\,.
 	\end{cases}$ 	
 \item[(v)]If $(\beta_1,\beta_2)\in  \cR_3$, then $\morse(S_T)=\begin{cases}
 		 			4,&\text{if }0<T<T_-(\beta_1,\beta_2)\,,\\
 			6,&\text{if }T_-(\beta_1,\beta_2)<T<T_+(\beta_1,\beta_2)\,,\\
 			4,&\text{if }T>T_+(\beta_1,\beta_2)\,.
 		\end{cases}$ 		
 	 		\end{enumerate}	}

 \end{proposition}

In order to prove Proposition \ref{prop0} we need to apply Theorem \ref{DeG} on the explicit expression of the characteristic polynomial of $S_T$. We obtain this polynomial below with some help from Lemma \ref{lem0}  and Corollary \ref{cor221} in the Appendix.

\begin{lemma}\label{lem4}
	{In the planar situation, the following equality holds true for every $T>0$: $$\det(S_T-\lambda I_8)=(d_4\lambda^4+d_3\lambda^3+d_2\lambda^2+d_1\lambda+d_0)^2,$$
			where $$\begin{cases}
				d_4:=1,\\ d_3:=c_1(\beta_1,\beta_2)\frac{T}{2\pi},\\  d_2:=\big(c_0(\beta_1,\beta_2)+3c_1(\beta_1,\beta_2)-8\big)\frac{T^2}{4\pi^2}-2,\\  d_1:=\big(c_0(\beta_1,\beta_2)+c_1(\beta_1,\beta_2)-4\big)\frac{T^3}{4\pi^3}-c_1(\beta_1,\beta_2)\frac{T}{2\pi},\\   d_0:=c_0(\beta_1,\beta_2)\frac{T^4}{16\pi^4}-c_1(\beta_1,\beta_2)\frac{T^2}{4\pi^2}+1,
			\end{cases}$$
and $c_0(\beta_1,\beta_2):=\beta_1\beta_2$;\ \
$c_1(\beta_1,\beta_2):=\beta_1+\beta_2+4$.

	}

  \end{lemma}
	\begin{proof}
Corollary \ref{cor221} gives 
		$$\det(S_T-\lambda I_8)=p_{T}(\lambda)^2\,,$$
		where 
 \begin{equation*}
 		\renewcommand*{\arraystretch}{1.4} 
    p_T(\lambda)=\det\left(-\frac{T}{2\pi}A+i J_2-\lambda I_4\right)=\det\left(\begin{array}{c|c}-(\frac{T}{2\pi}+\lambda)I_2&-\frac{T}{2\pi}\alpha_2-iI_2\\ \hline\frac{T}{2\pi}\alpha_2+iI_2&-\frac{T}{2\pi}W''(q_0)-\lambda I_2
		\end{array}\right).
\end{equation*}    
  Remembering that $W''(q_0)=V''(q_0)+I_2$ and using Lemma \ref{lem0} we see that
		\begin{equation*}
	p_T(\lambda)=\det\left(\left(\lambda^2+\frac{T}{\pi}\lambda -1\right)I_2+\left(\frac{T}{2\pi}\lambda+\frac{T^2}{4\pi^2}\right) V''(q_0)+i\,\frac{T}{\pi}\alpha_2\right)\,.
		\end{equation*}

The Hessian matrix $V''(q_0)$ is symmetric, and we deduce that, after possibly changing the order of the $\beta_i$'s there exists a rotation $\mathscr R\in SO(2)$ such that $\mathscr R^T V''(q_0)\mathscr R=\begin{pmatrix}
\beta_1&0 \\ 0&\beta_2
\end{pmatrix}$. The matrices $\alpha_2$ and $\mathscr R$ commute, and therefore
\begin{multline*}
p_T(\lambda)=\det\left(\left(\lambda^2+\frac{T}{\pi}\lambda -1\right)I_2+\left(\frac{T}{2\pi}\lambda+\frac{T^2}{4\pi^2}\right) \mathscr R^T V''(q_0)\mathscr R+i\,\frac{T}{\pi}\alpha_2\right)=\\=\det\begin{pmatrix}
\lambda^2+\frac{(2+\beta_1)T}{2\pi}\lambda +\frac{\beta_1T^2}{4\pi^2}-1& -\frac{T}{\pi} i\smallskip\\ \frac{T}{\pi} i&\lambda^2+\frac{(2+\beta_2)T}{2\pi}\lambda +\frac{\beta_2T^2}{4\pi^2}-1
\end{pmatrix}\,,
\end{multline*}
implying the result.
	\end{proof}

Given $T>0$ such that $\det(S_T)\not=0$, the negative Morse index $\morse(S_T)$ equals the order $4N=8$ minus the total algebraic multiplicity of the positive roots of the characteristic polynomial of $S_T$. By combining Lemma \ref{lem4} with Theorem \ref{DeG} we see that $\morse(S_T)$ equals $8$ minus twice the number of sign changes in the ordered sequence $\{d_4,d_3,d_2,d_1,d_0\}$. In particular, it depends only on $T$ and the (unordered) eigenvalues $\beta_1,\beta_2$. For this reason, in the following we shall assume, without loss of generality, that $V''(q_0)=\begin{pmatrix}
\beta_1&0\\ 0&\beta_2
\end{pmatrix}$ is diagonal, and write $S_T(\beta_1,\beta_2)$ instead of $S_T$.

 \smallskip

\begin{proof}
	[Proof of Proposition \ref{prop0}] In view of  Corollary \ref{corolnew} and Lemmas \ref{lem1234321}-\ref{lem6}, the set
$$\Gamma=\left\{S_T(\beta_1,\beta_2):(\beta_1,\beta_2)\in\bR^2,\ T>0,\ \det(S_T(\beta_1,\beta_2))\not=0\right\}\subset\mathbb R^{8\times 8}$$
can be written as the union of three connected components $\Gamma=\Gamma_1\cup\Gamma_2\cup\Gamma_3$. For instance, $\Gamma_1$ may be taken as the set of matrices $S_T(\beta_1,\beta_2)$ such that, either: $(\beta_1,\beta_2)\in\cR_0\cup(C\cap(\partial\cR_0))$ and $T>0$, or $(\beta_1,\beta_2)\in(\partial\cR_0)\backslash C$ and $T\not=T_+(\beta_1,\beta_2)$, or $(\beta_1,\beta_2)\in\cR_3$ and $T>T_+(\beta_1,\beta_2)$, or $(\beta_1,\beta_2)\in\bR^2\backslash\bar\cR_0$ and $T<T_-(\beta_1,\beta_2)$. Similarly, one may take  
$$\begin{cases}
	\Gamma_2=\left\{S_T(\beta_1,\beta_2):(\beta_1,\beta_2)\in\cR_1,\ T>T_+(\beta_1,\beta_2)\right\},\\
	\Gamma_3=\left\{S_T(\beta_1,\beta_2):(\beta_1,\beta_2)\in\bR^2\backslash\bar\cR_0,\ T_-(\beta_1,\beta_2)<T<T_+(\beta_1,\beta_2)\right\},	
\end{cases}$$
where $T_+(\beta_1,\beta_2):=+\infty$ if $(\beta_1,\beta_2)\in\cR_2\cup\cR_4\cup (C\backslash(\partial\cR_0))$.
Moreover, Lemmas \ref{lem1234321}-\ref{lem6} imply that $\Gamma$ does not contain singular matrices, and we deduce that the Morse index has a  constant value on each $\Gamma_i$. All what remains to do is to pick some matrix in each set and to compute its Morse index.

\medbreak

In this task, the computation of the characteristic polynomial of $S_T(\beta_1,\beta_2)$ carried out in Lemma \ref{lem4}  will be useful. We start with the choice $(\beta_1,\beta_2):=(1,1)\in\cR_1$ Notice that $c_0(1,1)=1$ and $c_1(1,1)=6$, and so
$\det(S_T(1,1)-\lambda I_8)=(d_4\lambda^4+d_3\lambda^3+d_2\lambda^2+d_1\lambda+d_0)^2\,,$
the coefficients $d_i$ being given by
$$d_4=1,\quad d_3=\frac{3T}{\pi},\quad d_2=\frac{11T^2}{4\pi^2}-2,\quad d_1=\frac{3T^3}{4\pi^3}-\frac{3T}{\pi},\quad d_0=\frac{T^4}{16\pi^4}-\frac{3T^2}{2\pi^2}+1\,.$$
 For $T>0$ small enough, $d_4,d_3,d_0>0$ while $d_2,d_1<0$, so that there are two sign changes in the ordered list $d_4,d_3,d_2,d_1,d_0$ and (by Theorem \ref{DeG}) we deduce that $\morse(S_T(\beta_1,\beta_2))=8-4=4$ if $S_T(\beta_1,\beta_2)\in \Gamma_1$. For $T>0$ big enough, all these coefficients are positive, there are no sign changes in the ordered list $d_4,d_3,d_2,d_1,d_0$ and therefore the positive Morse index of $S_T(1,1)$ is zero. Consequently, $\morse(S_T(\beta_1,\beta_2))=8-0=8$ if $S_T(\beta_1,\beta_2)\in \Gamma_2$. Similarly, one checks that the Morse index of $S_T(\beta_1,\beta_2)$ is $6$ on $\Gamma_3$, thus concluding the proof.
\end{proof}
We are now ready to study the Morse index of $S_T$ in the spatial case. In this situation $S_T$ is a $12\times 12$ symmetric matrix  depending on the parameter $T>0$.  By combining Corollaries \ref{corol} and \ref{corolnew} we see that $\det S_T=0$ if and only if either $T=\frac{2\pi}{\sqrt{\beta_3}}$  or $\det\tilde S_T=0$. Here,
	$$\renewcommand*{\arraystretch}{1.4}\tilde S_T=\left(\begin{array}{c|c}
	-\frac{T}{2\pi}\tilde A&- J_2\\ \hline J_2&-\frac{T}{2\pi}\tilde A
\end{array}\right)\in\mathbb R^{8\times 8},\qquad \tilde A=\left(\begin{array}{c|c}
I_2&\alpha_2\\ \hline -\alpha_2&\widetilde  W''(\widetilde q_0)\end{array}\right)\in\mathbb R^{4\times 4},$$ 
and $\widetilde  W''(\widetilde q_0)=\widetilde V''(\widetilde q_0)+I_2\in\bR^{2\times 2}$. On the other hand, the Morse index $m^-(\tilde S_T)$ was already computed in Proposition \ref{prop0} as a function of $T>0$ and the eigenvalues $\beta_1,\beta_2$ of $\widetilde V''(q_0)$. Notice that 
 \begin{proposition}\label{propo2}
 	{Let $T>0$ be such that $\det(\tilde S_T)\not=0$. Then,
 	$$m^-(S_T)=\begin{cases}
 		m^-(\tilde S_T)+2&\text{ if }0<T<\frac{2\pi}{\sqrt{\beta_3}},\\
 		m^-(\tilde S_T)+4&\text{ if }T>\frac{2\pi}{\sqrt{\beta_3}}.	
 	\end{cases}$$
 }
\begin{proof}
Corollary \ref{cor221} states that $\det(S_T-\lambda I_{12})=p_T(\lambda)^2$, where,  	

	\begin{equation*}
			\renewcommand*{\arraystretch}{1.4}
		\begin{split}
			&p_T(\lambda):=\det\left(-\frac{T}{2\pi}A+i J_3-\lambda I_{6}\right)=\\ &\det\left(\begin{array}{c|c|c|c}
				-(\frac{T}{2\pi}+\lambda)I_2 &0_{2\times 1} &-\frac{T}{2\pi}\alpha_2 + iI_2 &0_{2\times 1}\\ 
				\hline
				0_{1\times 2}  &-\frac{T}{2\pi}-\lambda   &0_{1\times 2}  &i\\
				\hline
				\frac{T}{2\pi}\,\alpha_2-iI_2 &0_{2\times 1}&-\frac{T}{2\pi}\widetilde W''(q_0)-\lambda I_2  &0_{2\times 1}\\
				\hline
				0_{1\times 2}   &-i &0_{1\times 2}   &-\frac{\beta_3T}{2\pi}-\lambda
			\end{array}\right).
		\end{split}
	\end{equation*}
	Here we have used the form (\ref{A}) of the matrix $A$. Rearranging rows and columns we see that  
	\begin{equation*}
			\renewcommand*{\arraystretch}{1.4}
		p_T(\lambda)=\det\left(\begin{array}{c|c|c|c}
			-(\frac{T}{2\pi}+\lambda)I_2 &-\frac{T}{2\pi}\alpha_2+ iI_2  &0_{2\times 1}&0_{2\times 1}\\
			\hline
			\frac{T}{2\pi}\,\alpha_2-iI_2   &-\frac{T}{2\pi}\widetilde W''(q_0)-\lambda I_2    &0_{2\times 1}  &0_{2\times 1}\\
			\hline
			0_{1\times 2}&0_{1\times 2}&-\frac{T}{2\pi}-\lambda& i\\
			\hline
			0_{1\times 2}   &0_{1\times 2}  &-i &-\frac{\beta_3T}{2\pi}-\lambda
		\end{array}\right).
\end{equation*}
The four blocks in the upper-left corner constitute the matrix $-\frac{T}{2\pi}\tilde A+iJ_2-\lambda I_4$, whose determinant will be denoted by $\widetilde p_T(\lambda)$.   On the other hand, the determinant of the $2\times 2$ matrix in the lower-right corner is $\lambda^2+\frac{(\beta_3+1)T}{2\pi}\,\lambda+\frac{\beta_3T^2-4\pi^2}{4\pi^2}$. It follows that
$$p_T(\lambda)=\widetilde p_T(\lambda)\left(\lambda^2+\frac{(\beta_3+1)T}{2\pi}\,\lambda+\frac{\beta_3T^2-4\pi^2}{4\pi^2}\right).$$

	On the other hand, Corollary \ref{cor221} states that $\det(\widetilde S_T-\lambda I_8)=\widetilde p_T(\lambda)^2$, leading us to the following connection between the characteristic polynomials of $S_T$ and $\widetilde S_T$:
	$$\det(S_T-\lambda I_{12})=\det(\widetilde S_T-\lambda I_8)\left(\lambda^2+\frac{(\beta_3+1)T}{2\pi}\,\lambda+\frac{\beta_3T^2-4\pi^2}{4\pi^2}\right)^2.$$
	In particular, the Morse index of $S_T$ equals the Morse index of $\tilde S_T$ plus twice the number of negative roots of the quadratic polynomial $\lambda^2+\frac{(\beta_3+1)T}{2\pi}\,\lambda+\frac{\beta_3T^2-4\pi^2}{4\pi^2}$. This latter number can be either computed directly or using Theorem \ref{DeG}, and the result follows.	
\end{proof}

 \end{proposition}

\section{Bifurcation numbers. Proofs of the main results}\label{sec:prf}
In this section we shall complete the proofs of Theorems \ref{thm:main2d}-\ref{thm:main3d}. It will be done by combining  Theorem \ref{th:bifurcation} with the results of Section \ref{sec:hamiltonian} and the explicit computation  of the bifurcation number associated to any trivial closed orbit. The bifurcation number has been defined in (\ref{bb3}) for stationary solutions of general Hamiltonian systems, but we shall now calculate them when the Hamiltonian has the form  \eqref{H}.

\smallskip

Thus, let $q_0\in\Omega$ be a critical point of $V$ and set $u_0:=(-\alpha_N q_0,q_0)$, which is a critical point of $H$. Given $T>0$, the bifurcation index of the corresponding closed orbit will be denoted, for simplicity, $\gamma_N(T,q_0):=\gamma_H(T,u_0)$. These are the bifurcation indexes mentioned at the end of Section \ref{sec:def}. In the following result we compute them in all possible situations, both in the planar and the spatial cases.
\begin{lemma}\label{biff} Let $(T,q_0)\in(0,+\infty)\times(V')^{-1}(0)$ be a trivial closed orbit. In the planar case {\bf [2d]}, the bifurcation number $\gamma_2(T,q_0)$ is given as follows:
	\begin{equation}\label{bn2}
	\gamma_2(T,q_0):=\begin{cases}
		-1&\text{if }(\beta_1,\beta_2)\in\cR_2\cup\cR_4\text{ and }T=T_-(\beta_1,\beta_2),\\
		1&\text{if }(\beta_1,\beta_2)\in\cR_1\cup\cR_3\text{ and }T=T_-(\beta_1,\beta_2),\\	
		\ib(q_0,V')&\text{if }(\beta_1,\beta_2)\in C\backslash(\partial\cR_0)\text{ and }T=T_-(\beta_1,\beta_2),\\	
		1&\text{if }(\beta_1,\beta_2)\in\cR_1\text{ and }T=T_+(\beta_1,\beta_2),\\
		-1&\text{if }(\beta_1,\beta_2)\in\cR_3\text{ and }T=T_+(\beta_1,\beta_2),\\
		0&\text{otherwise}.	
	\end{cases}
\end{equation}
	In the spatial case {\bf [3d]}, the bifurcation number $\gamma_3(T,q_0)$ is given by:
	\begin{equation}\label{bn3}
	\gamma_3(T,q_0):=\begin{cases}\gamma_2(T,\widetilde q_0)-1&\text{if }(\beta_1,\beta_2)\in\cR_2\cup\cR_4\text{ and }T=\frac{2\pi}{\sqrt{\beta_3}},\\
		\gamma_2(T,\widetilde q_0)+1&\text{if }(\beta_1,\beta_2)\in(\bar\cR_0\backslash C)\cup\cR_1\cup\cR_3\text{ and }T=\frac{2\pi}{\sqrt{\beta_3}},\\
		\gamma_2(T,\widetilde q_0)+\ib(\widetilde q_0,\widetilde V')&\text{if }(\beta_1,\beta_2)\in C\text{ and }T=\frac{2\pi}{\sqrt{\beta_3}},\\
		\gamma_2(T,\widetilde q_0)&\text{if }T\not=\frac{2\pi}{\sqrt{\beta_3}}.
	\end{cases}
\end{equation} 
\begin{proof}
	Just combine  the definition of bifurcation numbers \eqref{bb3} with the form of $\ib(u_0,H')$ obtained in Lemma \ref{lem2} (remember also \eqref{bd}), and the computation of the Morse indexes carried out in Propositions \ref{prop0}-\ref{propo2}. 
\end{proof}
\end{lemma}

We are now ready to complete the proof of the results announced in Section \ref{sec:results}. We start with the proof of assertion {\em (i)} in Theorem \ref{thm:main2d}.

\begin{proof}[Proof of Theorem \ref{thm:main2d}] 
	{\em (i)}: The combination of Lemma \ref{lem6}{\em (i)} and the Hartman-Grobman theorem implies that (HS) does not have periodic solutions in a sufficiently small neighborhood of $q_0$. The result follows from  Lemma \ref{lem1}{\em (ii)}.
\end{proof}

It remains to establish Theorem \ref{thm:main3d} and assertions {\em (ii)-(iii)} of Theorem \ref{thm:main2d}. By remembering Lemma \ref{lem2} and in view of the bifurcation numbers computed above, we can combine all these statements into one single result. It is the following:

\begin{theorem}\label{th:gen}{Let $(T,q_0)\in(0,+\infty)\times(V')^{-1}(0)$ be a trivial closed orbit of (\ref{soe}). If $\gamma_N(T,q_0)\not=0$ then there is a branch of closed orbits of (\ref{soe}) emanating from  $(T,q_0)$.} 
\begin{proof}
It was seen in Lemma \ref{lem1}{\em (i)} that all critical points of $H$ are isolated. Theorem \ref{th:bifurcation} applies and provides the existence of a branch $\widehat\cB$ of closed orbits of (HS) bifurcating from $u_0$. Therefore, Lemma \ref{lem1}{\em (iii)} implies the existence of a branch $\cB$ of closed orbits of \eqref{soe} emanating from $q_0$. The proof is complete.
\end{proof}
\end{theorem}

\section{Appendix} \label{sec:appendix}
 In this final section we present  a couple of tricks allowing us to simplify some computations from linear algebra. Throughout what follows $N\in\mathbb N$ is an arbitrary natural number. Our first lemma is a general tool allowing us to reduce the order of some determinants; it is possibly known but we could not find a precise reference.
\begin{lemma}\label{lem0}
	{Let $B_1,B_2,B_3,B_4\in\bR^{N\times N}$ be square matrices such that $B_1B_2=B_2B_1$. Then,
		$$\det\left(\begin{array}{c|c}
		B_1&B_2\\ \hline B_3&B_4
		\end{array}\right)=\det(B_4B_1-B_3B_2)\,.$$
		\begin{proof}It suffices to prove this result when $B_1$ is nonsingular (in the general case it suffices to replace $B_1$ by $B_1+\epsilon I_N$ and take limits when $\epsilon\to 0$). By substracting to the last $N$ columns  the product of the first $N$ columns by $B_1^{-1}B_2$  we get
			\begin{multline*}	\renewcommand*{\arraystretch}{1.2}
		\det\left(\begin{array}{c|c}
		B_1&B_2\\ \hline B_3&B_4
		\end{array}\right)=\det\left(\begin{array}{c|c}
			B_1&0_N\\ \hline B_3&B_4-B_3B_1^{-1}B_2
			\end{array}\right)=\\=\det(B_4-B_3B_1^{-1}B_2)\det B_1=\\=\det(B_4B_1-B_3B_1^{-1}B_2B_1)\,,
			\end{multline*}
			and the result follows from the fact that $B_1$ and $B_2$ commute.
		\end{proof}
		
	}
\end{lemma}

An important role in this paper is played by the matrices which we called $S_T$ in Section \ref{sec:branches}. 
In the following lemma we provide a trick to simplify the computation of their determinants.
\begin{lemma}\label{xlem12}
Let $A\in\mathbb R^{2N\times 2N}$ be symmetric and let $S_T\in\bR^{4N\times 4N}$ be defined as in \eqref{tkla}. Then, for every $T>0$ one has 
	$$\det S_T=\det\left(\frac{T}{2\pi}J_NA-iI_{2N}\right)^2=\det\left(-\frac{T}{2\pi}A+iJ_N\right)^2.$$
  \end{lemma}
	\begin{proof} We first show the result for $T=2\pi$. In this case, $S_{2\pi}=\left(\begin{array}{c|c}
			-A& - J_N\\ \hline J_N&-A
		\end{array}\right)$. By adding to the first $N$ rows the product of $i$ times the last $N$ rows and then substracting to the last $N$ columns the product of $i$ times the first $N$ columns one gets
		\begin{multline*}
		\det S_{2\pi}=\det\left(\begin{array}{c|c} -A+i J_N&- J_N-iA\\ \hline J_N&-A\end{array}\right)=\det\left(\begin{array}{c|c} -A+i J_N& 0_N\\ \hline J_N&-A-i J_N\end{array}\right)=\\=\det(-A+i J_N)\det(-A-i J_N)=|\det(-A-i J_N)|^2=\det(-A+i J_N)^2\,,
		\end{multline*}
		the last equality coming from the fact that $A+i J_N$ is Hermitian so that its determinant is a real number (we have also used the identity $\det\bar M=\overline{\det M}$). Since $ J_N^2=-I_{2N}$ and $\det J_N=1$ it follows that
		$\det S_{2\pi}=\det(J_NA-iI_{2N})^2$, thus proving the result for $T=2\pi$.

\smallskip	
 In the general case we apply the above to the symmetric matrix $\frac{T}{2\pi} A$. The result follows. 
\end{proof}

We immediately obtain the following well-known result:
\begin{corollary}\label{corolnew}
	{For $T>0$, let $S_T$ be defined as in (\ref{tkla}). Then $\det S_T=0$ if and only if $\frac{2\pi}{T}i$ is an eigenvalue of $J_N A$.}	
\end{corollary}
Applying Lemma \ref{xlem12} to the matrix $A':=A+\frac{2\pi}{T} \lambda I_{2N}$ one obtains an equality simplifying the computation of the characteristic polynomial of $S_T$. Precisely:
\begin{corollary}\label{cor221}
{Under the assumptions of Lemma \ref{xlem12}, the characteristic polynomial of $S_T$ is given by
	$$\det(S_T-\lambda I_{4N})=p_T(\lambda)^2\,,$$
where $p_T(\lambda):=\det\big(-\frac{T}{2\pi}A+i J_N-\lambda I_{2N}\big)$.}
\end{corollary}

\smallskip

\noindent
{\bf\em Acknowledgement.} We thank R. Ortega for pointing out references \cite{LerSch,Yor}.

\end{document}